\documentclass[12pt]{amsart}
\usepackage{amssymb,latexsym}
\usepackage{shuffle}
\usepackage{amscd}
\usepackage{difftrees}
\usepackage{mathtools}
\usepackage[all]{xy}
\usepackage[utf8]{inputenc}
\usepackage[top=35mm, bottom=25mm, left=30mm, right=30mm]{geometry}

\begin{document}

\author{Gunnar Fl{\o}ystad}
\address{Matematisk institutt\\
         Realfagsbygget \\
         Allegaten 41 \\
         5008 Bergen}
\email{gunnar@mi.uib.no}
%\urladdr{http://firstauthorwebaddress}
%\thanks{The firstauthor  ... thanks} 
\author{Hans Munthe-Kaas}
\address{Matematisk institutt \\
         Realfagsbygget\\
         Allegaten 41\\
         5008 Bergen}
\email{Hans.Munthe-Kaas@uib.no}
%\urladdr{http://secondauthorwebaddress}
%\thanks{The second author ... thanks}
% End two authors

%\keywords{keywords}
%\subjclass[2000]{Primary: subject; Secondary: subject}
\date{date}

% Two authors
\title[Pre- and post-Lie algebras: The algebro-geometry view]
{Pre- and post-Lie algebras:\\ The algebro-geometric view}

\begin{abstract}
We relate composition and substitution in pre- and
post-Lie algebras to algebraic geometry. The Connes-Kreimer Hopf algebras, 
and MKW Hopf algebras are
then coordinate rings of the infinite-dimensional affine varieties 
consisting of series of trees, resp.\ Lie series of ordered trees.
Furthermore we describe the Hopf algebras which are coordinate
rings of the automorphism groups of these varieties, 
which govern the substitution law in pre- and post-Lie algebras.
\end{abstract}
\maketitle

%\input /home/nmagf/prosjekt/2016/PostLie/macro.tex
% Option 5
% Theorems, corollaries, lemmas, and propositions, in the most 
% emphatic (plain) style; all are numbered separately.
% There is a Main Theorem in the most emphatic (plain) 
% style, unnumbered. There are definitions, in the less emphatic
% (definition) style. There are notations, in the least emphatic
%(remark) style, unnumbered.

\theoremstyle{plain}
\newtheorem{theorem}{Theorem}[section]
\newtheorem{corollary}[theorem]{Corollary}
\newtheorem*{main}{Main Theorem}
\newtheorem{lemma}[theorem]{Lemma}
\newtheorem{proposition}[theorem]{Proposition}

\theoremstyle{definition}
\newtheorem{definition}[theorem]{Definition}
\newtheorem{fact}[theorem]{Fact}

\theoremstyle{remark}
\newtheorem{notation}[theorem]{Notation}
\newtheorem{remark}[theorem]{Remark}
\newtheorem{example}[theorem]{Example}
\newtheorem{claim}{Claim}

%Mine egne.

\newcommand{\psp}[1]{{{\bf P}^{#1}}}
\newcommand{\psr}[1]{{\bf P}(#1)}
\newcommand{\op}{{\mathcal O}}
\newcommand{\opw}{\op_{\psr{W}}}
\newcommand{\go}{\op}

%Initial ideals
\newcommand{\ini}[1]{\text{in}(#1)}
\newcommand{\gin}[1]{\text{gin}(#1)}
\newcommand{\kr}{k}
\newcommand{\kkr}{{\Bbbk}}
\newcommand{\kk}{k}
\newcommand{\pd}{\partial}
\newcommand{\vardel}{\partial}
\renewcommand{\tt}{{\bf t}}

%Kategorier

\newcommand{\coh}{{{\text{{\rm coh}}}}}

%Modulkategorier

\newcommand{\modv}[1]{{#1}\text{-{mod}}}
\newcommand{\modstab}[1]{{#1}-\underline{\text{mod}}}

\newcommand{\sut}{{}^{\tau}}
\newcommand{\sumit}{{}^{-\tau}}
\newcommand{\til}{\thicksim}

\newcommand{\totp}{\text{Tot}^{\prod}}
\newcommand{\dsum}{\bigoplus}
\newcommand{\dprod}{\prod}
\newcommand{\lsum}{\oplus}
\newcommand{\lprod}{\Pi}

% Algebraer
\newcommand{\La}{{\Lambda}}

\newcommand{\sirstj}{\circledast}

% Knipper
\newcommand{\she}{\EuScript{S}\text{h}}
\newcommand{\cm}{\EuScript{CM}}
\newcommand{\cmd}{\EuScript{CM}^\dagger}
\newcommand{\cmri}{\EuScript{CM}^\circ}
\newcommand{\cler}{\EuScript{CL}}
\newcommand{\clerd}{\EuScript{CL}^\dagger}
\newcommand{\clerri}{\EuScript{CL}^\circ}
\newcommand{\gor}{\EuScript{G}}
\newcommand{\cF}{\mathcal{F}}
\newcommand{\cG}{\mathcal{G}}
\newcommand{\cH}{\mathcal{H}}
\newcommand{\cM}{\mathcal{M}}
\newcommand{\cE}{\mathcal{E}}
\newcommand{\cD}{\mathcal{D}}
\newcommand{\cI}{\mathcal{I}}
\newcommand{\cP}{\mathcal{P}}
\newcommand{\cK}{\mathcal{K}}
\newcommand{\cL}{\mathcal{L}}
\newcommand{\cS}{\mathcal{S}}
\newcommand{\cC}{\mathcal{C}}
\newcommand{\cO}{\mathcal{O}}
\newcommand{\cJ}{\mathcal{J}}
\newcommand{\cU}{\mathcal{U}}
\newcommand{\cR}{\mathcal{R}}
\newcommand{\cQ}{\mathcal{Q}}
\newcommand{\cX}{\mathcal{X}}
\newcommand{\mm}{\mathfrak{m}}

\newcommand{\dlim} {\varinjlim}
\newcommand{\ilim} {\varprojlim}

%Kategorier
\newcommand{\CM}{\text{CM}}
\newcommand{\Mon}{\text{Mon}}

%Kategorieer av komplekser

\newcommand{\Kom}{\text{Kom}}

% Begreper homologisk alebra

\newcommand{\EH}{{\mathbf H}}
\newcommand{\res}{\text{res}}
\newcommand{\Hom}{\text{Hom}}
\newcommand{\inhom}{{\underline{\text{Hom}}}}
\newcommand{\Ext}{\text{Ext}}
\newcommand{\Tor}{\text{Tor}}
\newcommand{\ghom}{\mathcal{H}om}
\newcommand{\gext}{\mathcal{E}xt}
\newcommand{\id}{\text{{id}}}
\newcommand{\im}{\text{im}\,}
\newcommand{\codim} {\text{codim}\,}
\newcommand{\resol}{\text{resol}\,}
\newcommand{\rank}{\text{rank}\,}
\newcommand{\lpd}{\text{lpd}\,}
\newcommand{\coker}{\text{coker}\,}
\newcommand{\supp}{\text{supp}\,}
\newcommand{\Ad}{A_\cdot}
\newcommand{\Bd}{B_\cdot}
\newcommand{\Fd}{F_\cdot}
\newcommand{\Gd}{G_\cdot}

%Avbildninger og andre symbolforkortelser

\newcommand{\sus}{\subseteq}
\newcommand{\sups}{\supseteq}
\newcommand{\pil}{\rightarrow}
\newcommand{\vpil}{\leftarrow}
\newcommand{\rpil}{\leftarrow}
\newcommand{\lpil}{\longrightarrow}
\newcommand{\inpil}{\hookrightarrow}
\newcommand{\pils}{\twoheadrightarrow}
\newcommand{\projpil}{\dashrightarrow}
\newcommand{\dotpil}{\dashrightarrow}
\newcommand{\adj}[2]{\overset{#1}{\underset{#2}{\rightleftarrows}}}
\newcommand{\mto}[1]{\stackrel{#1}\longrightarrow}
\newcommand{\vmto}[1]{\stackrel{#1}\longleftarrow}
\newcommand{\mtoelm}[1]{\stackrel{#1}\mapsto}

\newcommand{\eqv}{\Leftrightarrow}
\newcommand{\impl}{\Rightarrow}

\newcommand{\iso}{\cong}
\newcommand{\te}{\otimes}
\newcommand{\into}[1]{\hookrightarrow{#1}}
\newcommand{\ekv}{\Leftrightarrow}
\newcommand{\equi}{\simeq}
\newcommand{\isopil}{\overset{\cong}{\lpil}}
\newcommand{\equipil}{\overset{\equi}{\lpil}}
\newcommand{\ispil}{\isopil}
\newcommand{\vvi}{\langle}
\newcommand{\hvi}{\rangle}
\newcommand{\susneq}{\subsetneq}
\newcommand{\sgn}{\text{sign}}

%Notasjonsforkortelser

\newcommand{\xd}{\check{x}}
\newcommand{\ortog}{\bot}
\newcommand{\tL}{\tilde{L}}
\newcommand{\tM}{\tilde{M}}
\newcommand{\tH}{\tilde{H}}
\newcommand{\tvH}{\widetilde{H}}
\newcommand{\tvh}{\widetilde{h}}
\newcommand{\tV}{\tilde{V}}
\newcommand{\tS}{\tilde{S}}
\newcommand{\tT}{\tilde{T}}
\newcommand{\tR}{\tilde{R}}
\newcommand{\tf}{\tilde{f}}
\newcommand{\ts}{\tilde{s}}
\newcommand{\tp}{\tilde{p}}
\newcommand{\tr}{\tilde{r}}
\newcommand{\tfst}{\tilde{f}_*}
\newcommand{\empt}{\emptyset}
\newcommand{\bfa}{{\mathbf a}}
\newcommand{\bfb}{{\mathbf b}}
\newcommand{\bfd}{{\mathbf d}}
\newcommand{\bfl}{{\mathbf \ell}}
\newcommand{\bfx}{{\mathbf x}}
\newcommand{\bfm}{{\mathbf m}}
\newcommand{\bfv}{{\mathbf v}}
\newcommand{\bft}{{\mathbf t}}
\newcommand{\la}{\lambda}
\newcommand{\bfen}{{\mathbf 1}}
\newcommand{\bfe}{{\mathbf 1}}
\newcommand{\ep}{\epsilon}
\newcommand{\en}{r}
\newcommand{\tu}{s}
\newcommand{\Sym}{\text{Sym}}

\newcommand{\ome}{\omega_E}

\newcommand{\bevis}{{\bf Proof. }}
\newcommand{\demofin}{\qed \vskip 3.5mm}
\newcommand{\nyp}[1]{\noindent {\bf (#1)}}
\newcommand{\demo}{{\it Proof. }}
\newcommand{\demodone}{\demofin}
\newcommand{\parg}{{\vskip 2mm \addtocounter{theorem}{1}  
                   \noindent {\bf \thetheorem .} \hskip 1.5mm }}

\newcommand{\lcm}{{\text{lcm}}}

% Simplisielle komplekser

\newcommand{\dl}{\Delta}
\newcommand{\cdel}{{C\Delta}}
\newcommand{\cdelp}{{C\Delta^{\prime}}}
\newcommand{\dlst}{\Delta^*}
\newcommand{\Sdl}{{\mathcal S}_{\dl}}
\newcommand{\lk}{\text{lk}}
\newcommand{\lkd}{\lk_\Delta}
\newcommand{\lkp}[2]{\lk_{#1} {#2}}
\newcommand{\del}{\Delta}
\newcommand{\delr}{\Delta_{-R}}
\newcommand{\dd}{{\dim \del}}
\newcommand{\Del}{\Delta}

%Monomialidealer
\renewcommand{\aa}{{\bf a}}
\newcommand{\bb}{{\bf b}}
\newcommand{\cc}{{\bf c}}
\newcommand{\xx}{{\bf x}}
\newcommand{\yy}{{\bf y}}
\newcommand{\zz}{{\bf z}}
\newcommand{\mv}{{\xx^{\aa_v}}}
\newcommand{\mF}{{\xx^{\aa_F}}}

%Standard notasjoner
\newcommand{\Symm}{\text{Sym}}
\newcommand{\pnm}{{\bf P}^{n-1}}
\newcommand{\opnm}{{\go_{\pnm}}}
\newcommand{\ompnm}{\omega_{\pnm}}

\newcommand{\pn}{{\bf P}^n}
\newcommand{\hele}{{\mathbb Z}}
\newcommand{\nat}{{\mathbb N}}
\newcommand{\rasj}{{\mathbb Q}}
\newcommand{\bfone}{{\mathbf 1}}

\newcommand{\dt}{\bullet}
\newcommand{\disk}{\scriptscriptstyle{\bullet}}

\newcommand{\cxF}{F_\dt}
\newcommand{\pol}{f}

\newcommand{\Rn}{{\mathbb R}^n}
\newcommand{\An}{{\mathbb A}^n}
\newcommand{\frg}{\mathfrak{g}}
\newcommand{\PW}{{\mathbb P}(W)}

%Hibi
\newcommand{\pos}{{\mathcal Pos}}
\newcommand{\g}{{\gamma}}

%LPres
\newcommand{\Vaa}{V_0}
\newcommand{\Bp}{B^\prime}
\newcommand{\Bpp}{B^{\prime \prime}}
\newcommand{\bbp}{\mathbf{b}^\prime}
\newcommand{\bbpp}{\mathbf{b}^{\prime \prime}}
\newcommand{\bp}{{b}^\prime}
\newcommand{\bpp}{{b}^{\prime \prime}}

\def\CC{{\mathbb C}}
\def\GG{{\mathbb G}}
\def\ZZ{{\mathbb Z}}
\def\NN{{\mathbb N}}
\def\RR{{\mathbb R}}
\def\OO{{\mathbb O}}
\def\QQ{{\mathbb Q}}
\def\VV{{\mathbb V}}
\def\PP{{\mathbb P}}
\def\EE{{\mathbb E}}
\def\FF{{\mathbb F}}
\def\AA{{\mathbb A}}

\newcommand{\Lie}{\text{Lie}}
\newcommand{\hLie}{\widehat{\Lie}}
\newcommand{\Magma}{\text{Magma}}
\newcommand{\fg}{L}
\newcommand{\fc}{K}
\newcommand{\ofg}{\overline{\fg}}
\newcommand{\gr}{\text{gr}\,}
\newcommand{\grfg}{\gr \fg}
\newcommand{\Po}{P}
\newcommand{\CPo}{Q}
\newcommand{\Pre}{A}
\newcommand{\PL}{\text{PL}}
\newcommand{\OT}{\text{OT}}
\newcommand{\OF}{\text{OF}}
\newcommand{\surpil}{\twoheadrightarrow}
\newcommand{\oU}{\overline{U}}
\newcommand{\gd}{\circledast}
\newcommand{\shu}{\shuffle}
\newcommand{\Set}{\mathbf{Set}}
\newcommand{\PostLie}{\mathbf{PostLie}}
\newcommand{\DAlg}{\mathbf{D-algebra}}
\newcommand{\postlie}{\text{postLie}}
\newcommand{\cA}{\mathcal{A}}
\newcommand{\bihom}[2]{\overset{#1}{\underset{#2}{\rightleftarrows}}}
\newcommand{\gft}{\rhd}
\newcommand{\pode}{\curvearrowright}
\newcommand{\rf}{\sharp}
\newcommand{\End}{\text{End}}
\newcommand{\sta}{\star}
\newcommand{\fin}{\text{fin}}
\newcommand{\GL}{\text{GL}}
\newcommand{\SL}{\text{SL}}
\newcommand{\vertex}{\text{vert}}
\newcommand{\AdmV}{\text{AdmV}}
\newcommand{\AdmP}{\text{AdmP}}
\newcommand{\rot}{\text{root}}
\newcommand{\LB}{\text{LB}}
\newcommand{\injpil}{\hookrightarrow}
\newcommand{\charac}{\text{char.}}

\newcommand{\oA}{\overline{A}}
\newcommand{\hE}{\hat{E}}
\newcommand{\hP}{\hat{P}}
\newcommand{\ben}{{\mathbf 1}}
\newcommand{\hU}{\hat{U}}
\newcommand{\hte}{\hat{\te}}
\newcommand{\lbra}{[\![}
\newcommand{\rbra}{]\!]}
\newcommand{\plcdot}{\mathrlap{{\hskip 0.6mm}\bullet}{+}}
\newcommand{\plstar}{\mathrlap{{\hskip 0.55mm}*}{+}}
\newcommand{\field}{\text{field}}
\newcommand{\flow}{\text{flow}}
\newcommand{\Chi}{\chi}
\newcommand{\cZ}{{\mathcal Z}}
%\newcommand{\cI}{{\mathcal I}}

%\makeatletter
%\newcommand\xleftrightarrow[2][]{%
%  \ext@arrow 9999{\longleftrightarrowfill@}{#1}{#2}}
%\newcommand\longleftrightarrowfill@{%
%  \arrowfill@\leftarrow\relbar\rightarrow}
%\makeatother

%%%% Ekstra ting
%\newcommand{\bfv}{{\mathbf v}}
%\newcommand{\bfx}{{\mathbf x}}
%\newcommand{\bfm}{{\mathbf m}}
%\newcommand{\cE}{{\mathcal E}}
%\newcommand{\cH}{{\mathcal H}}
%%%%%%%%%  Fra PostLie2

%\newcommand{\LB}{\text{LB}}

\tableofcontents

\section*{Introduction}
Pre-Lie algebras  were first introduced in two different papers  from 1963.  Murray Gerstenhaber~\cite{gerstenhaber1963cohomology} studies deformations of algebras and Ernest Vinberg~\cite{vinberg1963theory}
 problems in differential geometry.  The same year John Butcher~\cite{butcher1963coefficients} published the first in a series of papers studying algebraic structures of numerical integration, culminating in his seminal paper~\cite{butcher1972algebraic} where B-series, the convolution product and the antipode of the Butcher--Connes--Kreimer Hopf algebra are introduced.
 %, see~\cite{hairer1974butcher} for the standard definition of B-series in numerical analysis. 
 
Post-Lie algebras  are generalisations of pre-Lie algebras introduced in the last decade. Bruno Vallette~\cite{vallette2007homology} introduced the post-Lie operad as the Koszul dual  of
the commutative trialgebra operad. Simultaneously post-Lie algebras appear in the study of numerical integration on Lie groups and manifolds~\cite{munthe2003enumeration,MK-W}. In a differential geometric picture a pre-Lie algebra is the algebraic structure of the flat and torsion free connection on a locally Euclidean space, whereas post-Lie algebras appear naturally as the  algebraic structure of the flat, constant torsion connection given by the Maurer--Cartan form on a Lie group~\cite{munthe2013post}. Recently it is shown that the sections of an anchored vector bundle admits a post-Lie structure if and only if the bundle is an action Lie algebroid~\cite{munthe16anchor}. 

B-series is a fundamental tool in the study of flow-maps (e.g.\ numerical integration) on euclidean spaces. The generalised Lie-Butcher LB-series are combining B-series with Lie series and have been introduced for studying integration on Lie groups and manifolds. 

In this paper we study B-series and LB-series from an algebraic geometry point of view. The space of B-series and LB-series can be defined as completions of the free pre- and post-Lie algebras. We study (L)B-series as an algebraic variety, where the coordinate ring has a natural Hopf algebra structure. In particular we are interested in the so-called substitution law. Substitutions for pre-Lie algebras were first introduced in numerical analysis~\cite{chartier2005}. The algebraic structure of pre-Lie substitutions and the underlying substitution Hopf algebra were introduced in~\cite{EF:TwoInteract}. For the post-Lie case, recursive formulae for substitution were given in~\cite{MK-L:Back}.  However, the corresponding Hopf algebra of substitution for post-Lie algebras was not  understood at that time. 

In the present work we show that the algebraic geometry view gives a natural way to understand both the Hopf algebra of composition and the Hopf algebra of substitution for pre- and post-Lie algebras. 

The paper is organised as follows. In Part 1 we study fundamental algebraic properties of the enveloping algebra of Lie-, pre-Lie and post-Lie algebras
%GF
for the general setting that these algebras $A$ are endowed with a
decreasing filtration $A = A^1 \supseteq A^2 \supseteq \cdots$. This seems
to be the general setting where we can define the exponential and
logarithm maps, and define the (generalized) Butcher product for pre- and
post-Lie algebras.
 Part~2 elaborates an algebraic geometric setting, where the pre- or post-Lie algebra forms an algebraic variety and the corresponding coordinate ring acquires  the structure of a Hopf algebra. This yields the Hopf algebra of substitutions in the free post-Lie algebra. Finally, we provide a recursive formula for the coproduct in this substitution Hopf algebra.

\part{The non-algebro geometric setting}
In this part we have no type of finiteness condition on
the Lie algebras, and pre- and post-Lie algebras.

\section{The exponential and logarithm maps for Lie algebras}

We work in the most general setting where we can define the 
exponential and logarithm maps. 
In Subsection \ref{subsec:ExpLogFil}
we assume the Lie algebra comes
with a decreasing filtration, and is complete
with respect to this filtration. We define the completed
enveloping algebra, and
discuss its properties. This is the natural general setting
for the exponential and logarithm maps which we recall in 
Subsection \ref{subsec:ExpLogMaps}.

\subsection{The Euler idempotent}

%We work with modules over a commutative ring $\kr$ of
%characteristic zero, i.e. it contains the field $\QQ$ of
%rational numbers. 
The setting in this subsection is any Lie algebra $L$, 
finite
or infinite dimensional over a field $\kr$ of characteristic 
zero. 
Let $U(L)$ be its enveloping algebra. 
This is a Hopf algebra with unit $\eta$, counit $\epsilon$ and
coproduct 
\[ \Delta : U(L) \pil U(L) \te_\kr U(L) \]
defined by $\Delta(\ell) = 1 \te \ell + \ell \te 1$ for any
$\ell \in L$, and extended to all of $U(L)$ by requiring $\Delta$ to 
be an algebra homomorphism.

For any algebra $A$ with multiplication map $\mu_A : A \te A \pil A$,
we have the convolution  product $\star$ on $\Hom_\kr(U(L),A)$. For $f,g \in 
\Hom_\kr(U(L),A)$ it is defined as
\[ f \star g = \mu_A \circ (f \te g) \circ \Delta_{U(L)}. \]
Let $\ben$ be the identity map on $U(L)$, and $J = \ben - \eta \circ \epsilon$.
The {\it Eulerian idempotent} $e : U(L) \pil U(L)$ is defined by
\[ e = \log (\ben) = \log (\eta \circ \epsilon + J) 
= J - \frac{J^{\star 2}}{2} + \frac{J^{\star 3}}{3} - \cdots. \]

\begin{proposition} \label{pro:ExpLogEuler}
The image of $e: U(L) \pil U(L)$ is $L \sus U(L)$, and
$e$ is the identity restricted to $L$.
\end{proposition}

\begin{proof} This is a special case of the canonical decomposition
stated in 0.4.3 in \cite{Re}. See also Proposition
3.7, and part (i) of its proof in \cite{Re}.
\end{proof}

Let $\Sym^c(L)$ be the free cocommutative coalgebra on $L$.
It is the subcoalgebra of the tensor coalgebra $T^c(L)$ consisting
of the symmetrized tensors
\begin{equation} \label{eq:ExpLogSym}
\sum_{\sigma \in S_n} l_{\sigma(1)} \te l_{\sigma(2)} \te \cdots
l_{\sigma(n)} \in L^{\te n}, \quad l_1, \ldots l_n \in L.
\end{equation}
The above proposition gives a linear map $U(L) \mto{e} L$.
Since $U(L)$ is a cocommutative
coalgebra, there is then a homomorphism of cocommutative coalgebras
\begin{equation} U(L) \mto{\alpha} \Sym^c(L). \label{eq:ExpLogAlfa}
\end{equation}
We now have the following strong version of the Poincar{\'e}-Birkhoff-Witt
theorem.

\begin{proposition} \label{pro:ExpLogPBW}
The map $ U(L) \mto{\alpha} \Sym^c(L)$ is an isomorphism 
of coalgebras.
\end{proposition}

In order to show this we expand more on the Euler idempotent.

Again for $l_1, \ldots, l_n \in L$ denote by $(l_1, \ldots, l_n)$ the
symmetrized product in $U(L)$: 
\begin{equation} \label{eq:ExpLogSymprod}
 \frac {1}{n!} \sum_{\sigma \in S_n} l_{\sigma(1)}l_{\sigma(2)} \cdots
l_{\sigma(n)}, 
\end{equation}
and let $U_n(L) \sus U(L)$ be the subspace generated by all these
symmetrized products.

\begin{proposition} Consider the map given by convolution of the
Eulerian idempotent:
\[ \frac{e^{\star p}}{p!} : U(L) \pil U(L).\]
\begin{itemize}
\item [a.]The map above is zero on $U_q(L)$ when $q \neq p$ and, the 
identity on $U_p(L)$.
\item [b.] The sum of these maps 
\[ \exp(e) = \eta \circ \epsilon + e + \frac{e^{\star 2}}{2} + 
\frac{e^{\star 3}}{3!} + \cdots \]
is the identity map on $U(L)$. (Note that the map is well since
the maps $e^{\star p}/p!$ vanishes on any element in $U(L)$ for $p$ sufficiently
large.)
\end{itemize}
From the above we get a decomposition
\[ U(L) = \bigoplus_{n \geq 0} U_n(L). \]
\end{proposition}

\begin{proof} This is the canonical decomposition stated in
0.4.3 in \cite{Re}, see also Proposition 3.7 and its proof in \cite{Re}.
\end{proof}

\begin{proof}[Proof of Proposition \ref{pro:ExpLogPBW}]
Note that since $e$ vanishes on $U_n(L)$ for $n \geq 2$, by the way 
one constructs the map 
$\alpha$, it sends the symmetrizer $(l_1, \ldots, l_n) \in U_n(L)$
to the symmetrizer (\ref{eq:ExpLogSymprod}) in $\Sym_n^c(L)$.
This shows $\alpha$ is surjective. But there is also a linear map, the
surjective section
$\beta : \Sym_n^c(L) \pil U_n(L)$ sending the
symmetrizer (\ref{eq:ExpLogSymprod}) to the symmetric product $(l_1, \ldots,
l_n)$. This shows that $\alpha$ must also be injective.
\end{proof}

\subsection{Filtered Lie algebras} \label{subsec:ExpLogFil}
%Now the setting is that $\kr$ is a {\it local} commuative ring
%with maximal ideal $\mm$.
Now the setting is that the Lie algebra $\fg$ comes with a 
filtration
\[ \fg = \fg^1 \supseteq \fg^2 \supseteq \fg^3 \supseteq \cdots \]
such that $[\fg^i, \fg^j] \sus \fg^{i+j}$.
% and $\mm \fg^i \sus \fg^{i+1}$.
%It gives an associated graded Lie algebra $\grfg = \oplus_{i \geq 1} 
%\fg_i/\fg_{i+1}$.
Examples of such may be derived from any Lie algebra over $\kr$:
\begin{itemize}
\item[1.] The lower central series gives such a filtration with 
$\fg^2 = [\fg,\fg]$ and $\fg^{p+1} = [\fg^p, \fg]$. 
\item[2.] The polynomials $L[h] = \oplus_{n \geq 1} L h^n$.
%which is a Lie algebra over the
%graded (and hence local) ring $\kr = k[h]$.
\item[3.] The power series $L\lbra h \rbra = \Pi_{n \geq 1} Lh^n$.
%which is a Lie algebra over the power series ring $\kr = k\lbra h \rbra$.
\end{itemize}

Let $\Sym_n(L)$ be the symmetric product of $L$, that is
the natural quotient of $L^{\te n}$ which is the
coinvariants $(L^{\te n})^{S_n}$ for the action of the symmetric group $S_n$.
By the definition of $\Sym^c(L)$ in \eqref{eq:ExpLogSym} there are maps
\[ \Sym_n^c(L) \hookrightarrow L^{\te n} \pil \Sym_n(L),\]
and the composition is a linear isomorphism.
We get a filtration on $\Sym_n(L)$ by letting
\[ F^p(L) = \sum_{i_1 + \cdots + i_n \geq p} L^{i_1} \cdots L^{i_n}. \]
The filtration on $L$ gives an associated graded Lie algebra 
$\grfg = \oplus_{i \geq 1} \fg_i/\fg_{i+1}$. The filtration on
$\Sym_n(L)$ also induces an associated graded vector space. 

\begin{lemma} \label{lem:ExpLogSymgr}
There is an isomorphism of associated graded vector spaces
\begin{equation} \label{eq:ExpLogSnL} 
\Sym_n(\gr L) \mto{\iso} \gr \Sym_n(L).
\end{equation}
\end{lemma}

\begin{proof}
%GF-beviset er gjort detaljert
Note first that there is a natural map (where $d$ denotes the grading
induced by the graded Lie algebra $\gr L$) 
\begin{equation} \label{eq:ExpLogSymfilt}
\Sym_n(\gr L)_d \pil F^d \Sym_n(L)/F^{d+1} \Sym_n(L). 
\end{equation}
It is also clear by how the filtration is defined 
that any element on the right may be lifted to some element
on the left, and so this map is surjective. We must then show that it
is injective.

Choose splittings $L/L^{i+1} \mto{s_i} L$ of $L \pil L/L^{i+1}$ for 
$i = 1, \ldots
p$, and let $L_i = s_i(L^i/L^{i+1})$. Then we have a direct sum decomposition
\[ L = L_1 \oplus L_2 \oplus \cdots \oplus L_p \oplus \cdots \]
Since in general $\Sym_n(A \oplus B)$ is equal to $\oplus_{i} \Sym_{i}(A) \te
\Sym_{n-i}(B)$ we get that
\begin{equation} \label{eq:ExpLogSymdecomp}
\Sym_n(L) = \oplus_{i_1, \ldots, i_p} S_{i_1}(L_1) \te \cdots
\te S_{i_p}(L_p), 
\end{equation}
where we sum over all compositions where $\sum i_j = n$.

\noindent{\bf Claim.} 
\[ F^d S_n(L) = \oplus_{i_1,\ldots, i_p} S_{i_1}(L_1) 
\te \cdots \te S_{i_p}(L_p), \]
where we sum over all $\sum i_j = n$ and $\sum j \cdot i_j \geq d$. 

\noindent{\it Proof of Claim.} Clearly we have an inclusion  $\supseteq$. 
Conversely let $a \in F^d \Sym_n(L)$. Then $a$ is a sum of products
$a_{r_1} \cdots a_{r_q}$ where $a_{r_j} \in L^{r_j}$ and $\sum r_j \geq d$. 
But then each $a_{r_j} \in \oplus_{t \geq r_j} L_t$, and so by the direct sum
decomposition in \eqref{eq:ExpLogSymdecomp}, 
each $a_{r_1} \cdots a_{r_q}$ lives in the 
right side of the claimed equality, and so does $a$.

\medskip
It now follows that the two sides of \eqref{eq:ExpLogSymfilt} have
the same dimension and so this map is an isomorphism.
\end{proof}

%The filtration gives an associated graded Lie algebra 
%$\grfg = \oplus_{i \geq 1} 
%\fg_i/\fg_{i+1}$, and its enveloping algebra $U(\gr L)$.
We have the enveloping algebra $U(\fg)$ and the enveloping algebra
of the associated graded algebra $U(\grfg)$.
 The augmentation ideal
$U(\fg)_+$ is the kernel $\ker U(\fg) \mto{\epsilon} \kr$ of the counit.
The enveloping algebra $U(L)$ now gets a filtration of ideals by letting
$F^1 = U(\fg)_+$ and 
\[ F^{p+1} = F^p \cdot U(\fg)_+ + (L^{p+1}), \] where $(L^{p+1})$ is the
ideal generated by $L^{p+1}$.
%Let $F^p \sus U(\fg)$ be the ideal generated by all monomials
%$u_1\cdots u_r$ where $u_i \in F^{t_i}$ and $\sum_i t_i \geq p$.
%This gives a filtration of ideals of the enveloping algebra $U(\fg)$
%\[ U(\fg) = F^0 \supset F^1 \supset F^2 \supset \cdots .\]
This filtration induces again a graded algebra 
\[ \gr^{\oplus} U(\fg) = \bigoplus_i F^i/F^{i+1}. \]
There is also another version, the graded product algebra, which we
will encounter later
\[ \gr^{\Pi} U(\fg) = \prod_i F^i/F^{i+1}. \]

\begin{proposition} \label{pro:ExpLogNatgriso}
%Assume the filtration  of $\fg$ is exhaustive,
%i.e. $\cap_p \fg^p = 0$. Then the filtration $F^p$ is exhaustive,
%i.e. $\cap_p F^p = 0$
The natural map of graded algebras 
\[ U(\grfg) \mto{\iso} \gr U(\fg), \]
is an isomorphism.
\end{proposition}

\begin{proof}
The filtrations on each $\Sym_n^c(L)$ induces a filtration on
$\Sym^c(L)$. Via the isomorphism $\alpha$ of \eqref{eq:ExpLogAlfa}
and the explicit form given in the proof of Proposition \ref{pro:ExpLogPBW}
the filtrations on 
$U(L)$ and on $\Sym^c_n(L)$ correspond. Hence 
\[ \gr \alpha : \gr U(L) \mto{\iso} \gr \Sym^c_n(L) \]
is an isomorphism of vector spaces.
There is also an isomorphism $\beta$ and a commutative diagram
\[ \begin{CD}  U(\gr L) @>{\beta}>> \Sym^c(\gr L) \\
@VVV  @VVV \\
\gr U(L) @>{\gr \alpha}>> \gr \Sym^c(L).
\end{CD} \] 
By Lemma \ref{lem:ExpLogSymgr} the right vertical map is an isomorphism
and so also the left vertical map.
\end{proof}

%From now on we assume that these filtrations are exhaustive.
The cofiltration
\[  \cdots \surpil U(\fg)/F^{n} \surpil U(\fg)/F^{n-1} \surpil \cdots \]
induces the completion
\[ \hU(\fg) = \underset{p} \varprojlim \, U(\fg)/F^p. \]
This algebra also comes with the filtration $\hat{F}^p$. 
Let $\hat{\fg} = \underset{p} \varprojlim \, \fg/\fg^p$. 
\begin{lemma}
The completed algebras are equal:
\[ \hU(\hat{\fg}) = \hU(\fg), \]
and so this algebra only depends on the completion
$\hat{\fg}$.
\end{lemma}

\begin{proof}
The natural map $L \pil \hat{L}$ induces a natural map 
$U(\fg) \mto{\gamma} U(\hat{\fg})$.
Since $\fg$ and $\hat{\fg}$ have the same associated graded Lie
algebras, the two downwards maps in the commutative diagram
\[ \xymatrix{\gr U(\fg) \ar[dr ]  & \ar[r]   & \gr U(\hat{\fg})  \ar[ld] \\
                           & U(\gr \fg) &    } 
\]
are isomorphisms, showing that the upper horizontal map is an isomorphism.
But given the natural map $\gamma$ this easily implies that the map
of quotients
\[ U(\fg)/F^{p+1}U(\fg) \mto{\gamma^p} U(\hat{\fg})/F^{p+1}U(\hat{\fg}) \]
is an isomorphism, and so the completions are isomorphic.
\end{proof}

We denote the $d$'th graded part of the enveloping algebra $U(\grfg)$ 
by $U(\grfg)_d$. The following gives an idea of the ``size'' of 
$\hU(\fg)$. 

\begin{lemma} 
\[ \gr^{\Pi} \hU(\fg) = \hU(\grfg) = \prod_{d \in \hele} U(\grfg)_d. \]
\end{lemma}

\begin{proof}
The left graded product is 
\[ \gr^{\Pi} \hU(\fg) = \prod_{p \geq 0} F^p/F^{p+1}.\]
But by Proposition \ref{pro:ExpLogNatgriso} 
$F^p/F^{p+1} \iso U(\gr L)_p$ and so the above statement follows.
\end{proof}

\begin{example}
Let $V = \oplus_{i \geq 1} V_i$ 
be a graded vector space with $V_i$ of degree $i$,
and let $\Lie(V)$ be the free Lie algebra on $V$. It then has
a grading $\Lie(V) = \oplus_{d \geq 1} \Lie(V)_d$ coming from the grading on $V$,
and so a filtration $F^p = \oplus_{d \geq p} \Lie(V)_d$.
The enveloping algebra $U(\Lie(V))$ is the tensor algebra
$T(V)$.
%The completion of $\Lie(V)$ with respect to this filtration is
%\[ \hLie(V) = \prod_{d \geq 1} \Lie(V)_d. \]
The completed enveloping algebra is 
\[ \hU(\Lie(V)) = \hat{T}(V) := \prod_{d} T(V)_d. \]
\end{example}

Let $\fg_p$ be the quotient $\fg/\fg^{p+1}$, which is 
a nilpotent filtered Lie algebra. 
We get enveloping algebras $U(\fg_p)$ with filtrations
$F^jU(\fg_p)$ of ideals, and quotient algebras
\[U^j(\fg_p) = U(\fg_p)/F^{j+1}U(\fg_p). \]

\begin{lemma} \label{Lem:ExpLogLimjp}
%Assume $\fg$ is complete with respect to the filtration, i.e. $\hat{\fg} = 
%\fg$. Then 
\[  \hU(\fg) = \underset{j,p}{\varprojlim} \, U^j(\fg_p). \]
\end{lemma}

\begin{proof}
First note that if $j \leq p$ then $U^p(L_p) \twoheadrightarrow U^j(L_p)$
surjects. If $j \geq p$, then $U^j(L_j) \twoheadrightarrow U^j(L_p)$
surjects. Hence it is enough to show that the natural map
\[ U(L)/F^{p+1} \pil U(L_p)/F^{p+1}U(L_p) = U^p(L_p) \]
is an isomorphism. This follows since we have an isomorphism
of associated graded vector spaces:
\begin{align*}
(\gr (U(L)/F^{p+1}))_{\leq p} & = (\gr U(L))_{\leq p} \iso U(\gr L)_{\leq p} \\
    & = U(\gr L_p)_{\leq p} \iso (\gr U(L_p))_{\leq p} \\
    & = (\gr U(L_p)/F^{p+1})_{\leq p}
\end{align*}
\end{proof}

\subsection{The exponential and logarithm}
\label{subsec:ExpLogMaps}
%We may let $\fg_p = \fg/\fg^{p+1}$. This is a nilpotent filtered Lie algebra,
%and we have a filtration $F^j_p$ of $U(\fg_p)$, and a completion
%$\hU(\fg_p) = \underset{j} \varprojlim U(\fg_p)/F^j_p.$

The coproduct $\Delta$ on $U(\fg)$ will send
\[ F^p \mto{\Delta} 1 \te F^p + F^1 \te F^{p-1} + \cdots + F^p \te 1. \]
Thus we get a map 
\[ \hU(\fg) \pil U(\fg)/F^{2p-1} \mto{\Delta} U(\fg)/F^p \te U(\fg)/F^p. \]
Let 
\[ \hU(\fg) \hte \hU(\fg) := \underset{p}\varprojlim \,
U(\fg)/F^p \te U(\fg)/F^p \]
be the completed tensor product
We then get a {\it completed coproduct}
\[ \hU(\fg) \mto{\Delta} \hU(\fg) \hte \hU(\fg). \] Note 
that the tensor product 
\[\hU(\fg) \te \hU(\fg) \sus \hU(\fg) \hte \hU(\fg).\]

An element $g$ of $\hU(\fg)$ is {\it grouplike} if $\Delta(g) = g \te g$
in $\hU(\fg) \te \hU(\fg)$. We denote the set of grouplike elements
by $G(\hU(\fg))$. They are all of the form $1 + s$ where $s$ is in 
the augmentation 
\[ \hU(\fg)_+ = \ker (\hU(\fg) \mto{\epsilon} \kr). \]

\medskip The exponential map 
\[ \hU(\fg)_+ \mto{exp} 1 + \hU(\fg)_+ \] is given by 
\[ \exp(x) = 1 + x + \frac{x^2}{2!} + \frac{x^3}{3!} + \cdots. \]
%It is a grouplike element in $\hU(\fg)$. 

The logarithm map 
\[ 1 + \hU(\fg)_+ \mto{log} \hU(\fg)_+ \]
is defined by 
\[ \log(1+s) = s - \frac{s^2}{2} + \frac{s^3}{3} - \cdots . \]

\begin{proposition} \label{pro:ExpLogExpLog}
The maps
\[  \hU(\fg)_+ \bihom{\exp}{\log}  1 + \hU(\fg)_+ \]
give inverse bijections. They restrict to inverse bijections
\[ \hat{\fg} \bihom{\exp}{\log} G(\hU(\fg)) \]
between the completed Lie algebra and the grouplike elements. 
%\[ \fg \underset{\log}{\overset{\exp}{\bihom}{}{}} G(U(\fg)) \] 
\end{proposition}

\begin{proof}
That $\log(\exp(x))$ and $\exp(\log(1+s)) = 1+s$, are formal 
manipulations. If $\ell \in \hat{\fg}$ it is again a formal
manipulation that 
\[\Delta(\exp(\ell)) = \exp(\ell) \cdot \exp(\ell), \]
and so $\exp(\ell)$ is a grouplike element.

The maps $\exp$ and $\log$ can also be defined on the tensor products 
and give inverse
bijections
\[ \hU(\fg)_+ \hte \hU(\fg) + \hU(\fg) \hte \hU(\fg)_+ 
\bihom{\exp}{\log}  1 \te 1  + 
 \hU(\fg)_+ \hte \hU(\fg) +
\hU(\fg) \hte \hU(\fg)_+.
 \]
Now let $s \in G(\hU(\fg))$ be a grouplike element. Since 
$\Delta = \Delta_{\hU(L)}$
is an algebra homomorphism
\[ \exp(\Delta(\log(s))) = \Delta(\exp(\log s)) = \Delta(s) = s \te s. \]
Since $1 \te s$ and $s \te 1$ are commuting elements we also have
\[ \exp(\log(s) \te 1 + 1 \te \log(s)) = (\exp(\log(s)) \te 1) \cdot
(1 \te \exp(\log(s))) = s \te s.\]
Taking logarithms of these two equations, we obtain
\[ \Delta(\log(s)) = \log(s) \te 1 + 1 \te \log(s), \]
and so $\log(s)$ is in $\hat{\fg}$. 
\end{proof}

\section{Exponentials and logarithms for pre- and post-Lie 
algebras}  \label{sec:PrePo}

For pre- and post-Lie algebras their enveloping algebra comes with
two products $\bullet$ and $*$. This gives two possible exponential
and logarithm maps. This is precisely the setting that enables us
to define a map from formal vector fields to formal flows. It also
gives the general setting for defining the Butcher product.

\subsection{Filtered pre- and post-Lie algebras}
Given a linear binary operation on a $\kr$-vector space $A$
\[ * : A \te_\kr A \pil A \]
the associator is defined as:
\[ a_*(x,y,z) = x*(y*z) - (x*y)*z. \]

\begin{definition}  \label{def:PrePoRel}
A {\it post-Lie algebra} $(\Po, [,],\rhd)$ is a Lie algebra 
$(\Po,[,])$ together with a linear binary map 
$\rhd$ such that
\begin{itemize}
\item $x \rhd [y,z] = [x\rhd y,z] + [x, x\rhd z]$
\item $[x,y] \rhd z = a_{\rhd}(x,y,z) - a_{\rhd}(y,x,z)$
\end{itemize}
\end{definition}
It is then straightforward to verify that the following
bracket 
%HMK - endret fortegn siste Lie bracket
\[ \lbra x,y \rbra = x \rhd y  - y \rhd x + [x,y] \]
defines another Lie algebra structure on $\Po$.

A {\it pre-Lie} algebra is a post-Lie algebra $\Po$ such that bracket
$[,]$ is zero, so $\Po$ with this bracket is the abelian Lie algebra.

\medskip
We now assume that $\Po$ is a filtered post-Lie algebra: We have
a decreasing filtration
\[ \Po = \Po^1 \supseteq \Po^2 \supseteq \cdots , \]
such that 
\[ [\Po^p, \Po^q] \sus \Po^{p+q}, \quad
  \Po^p \rhd \Po^q \sus \Po^{p+q}, \]
Then we will also have $\lbra \Po^p, \Po^q \rbra \sus \Po^{p+q}$.
If $u$ and $v$ are two elements of $\Po$ such that $u-v \in \Po^{n+1}$,
we say they are equal up to order $n$.

Again examples of this can be constructed for any post-Lie algebra
over a field $k$ by letting $\Po^1 = \Po$ and  
\[ \Po^{p+1} := \Po^p \rhd \Po + \Po \rhd \Po^p + [\Po,\Po^p]. \]
Alternatively we may form the polynomials 
$P[h] = \oplus_{n \geq 1} Ph^n$,
or the power
series $P[[h]] = \Pi_{n \geq 1} Ph^n$.

\medskip In \cite{MK-L-EF:Env} the enveloping algebra $U(\Po)$ of the
post-Lie algebra was introduced. It is both the enveloping algebra
for the Lie algebra $[,]$ and as such comes with associative product
$\bullet$, and is the enveloping algebra for the Lie algebra $\lbra, \rbra$ and
as such comes with associative product $*$. The triangle product
also extends to a product $\rhd$ on $U(\Po)$ but this is not
associative.

\subsection{The map from fields to flows}

\begin{example}
Butcher and Lie-Butcher series
\end{example}

LB-series are elements of a complete filtered  post-Lie algebra $P$. B-series are the special case where $P$ is pre-Lie. 
Using the perspective of these series there are
several natural ways to think of these objects.
\begin{itemize}
\item The elements of $\Po$ may be thought of as formal vector fields,
in which case we write $\Po_{field}$.
\item The grouplike elements of $\hU(\Po)$ may be thought of
 as formal flows.
\item The elements of $\Po$ may be thought of as principal
parts of formal flows, see below, in which case we write $\Po_{flow}$.
\end{itemize}

Let us explain how these are related. In the rest of this subsection
we assume that $P = \hat{P}$ is complete with respect to the
filtration.
The exponential map
\[ \Po_{field} \mto{\exp^*} \hU(\Po) \]
sends a vector field to a formal flow, a grouplike element
in $\hU(\Po)$, which in the example above is a Lie-Butcher series.
(Note that the notion of a grouplike element in $\hU(\Po)$ only 
depends on the shuffle coproduct.)

We may take the logarithm
\[ G(\hU(\Po)) \mto{\log^{\bullet}} \Po.\]
So if $B \in G(\hU(\Po))$ we get $b = \log^{\bullet}(B)$.
We think of $b$ also as a formal flow, the {\it principal part} or 
{\it first order part} of the formal flow $B$. It determines $B$ by
$B = \exp^{\bullet}(b)$. 

When $\Po$ is a pre-Lie algebra $\Pre$, then $\hU(\Po)$ is the
completed symmetric algebra $\widehat{\Sym}(\Pre)$ and 
$\log^{\bullet}$ is simply the projection
$\widehat{\Sym}(\Pre) \pil \Pre$. If $B$ is a Butcher series parametrized
by forests, then  $b$ is the Butcher series parametrized by trees.
Thus $b$ determines the flow, but the  full series $B$ is necessary
to compute pull-backs of functions along the flow.

\medskip
We thus get a bijection
\begin{equation} \label{eq:PoPreVF}
\Phi : \Po_{field} \xrightarrow{\log^{\bullet} \circ \exp^*} \Po_{flow} 
\end{equation}
which maps vector fields to principal part flows. This map is closely related to the \emph{Magnus expansion}~\cite{ebrahimi2011twisted}. 
Magnus expresses the exact flow as $\exp^*(tv) = \exp^\bullet(\Phi(tv))$, from which a differential equation for $\Phi(tv)$ can be derived.

\begin{example} \label{exa:PrePoFlow} Consider the manifold $\RR^n$ and 
let $\cX \RR^n$ be the vector fields on $\RR^n$.
Let $f = \sum_{i \geq 0} f_ih^i$ on $\RR^n$
be a power series of  vector fields where
each $f_i \in \cX \RR^n$. 
It induces the flow series $\exp^*(hf)$ in 
$\hat{U}(\cX \RR^n  \lbra h \rbra)$. Since $\cX \RR^n$ is a
pre-Lie algebra, the completed
enveloping algebra is $\widehat{\Sym}(\cX \RR^n  \lbra h \rbra)$.
Thus the series
\[ \exp^*(hf) = 1 + \sum_{i \geq d \geq 1} F_{i,d}h^i \]
where the $F_{i,d} \in \Sym_d(\cX \RR^n \lbra h \rbra)$ are $d$'th
order differential operators. (Note that the principal part $b$ is the 
$d=1$ part.)
It determines a flow $\Psi^f_h: \RR^n \pil \RR^n$
sending a point $P$ to $P(h)$. For any smooth function $\phi: \RR^n \pil  \RR$
the {\it pullback} of $\phi$ along the flow is the composition
$\phi \circ \Psi^f_h : \RR^n \pil \RR$ and is given by 
\[ \exp^*(hf)\phi = 1 + \sum_{i \geq d \geq 1} F_{i,d}(\phi) h^i,  \]
see \cite[Section 4.1]{MK-L:Formal} or \cite[Section 2.1]{MK-F}.
In particular when $\phi$ is a coordinate function $x_p$ we get the
coordinate $x_p(h)$ of $P(h)$ as given by 
\[ x_p(h) = \exp^*(hf)x_p = 
\sum_{i \geq d \geq 1} F_{i,d} x_p h^i = x_p + \sum_{i} F_{i,1} x_p h^i \]
since higher derivatives of $x_p$ vanish. This shows concretely 
geometrically why the flow is determined by its principal part.
\end{example}

For a given principal flow $b \in \Po_{flow}$
computing its inverse image by the map 
\eqref{eq:PoPreVF} above, which is the vector field 
$\log^* \circ \exp^{\bullet}(b)$ is called {\it backward error} in 
numerical analysis~\cite{hairer1994backward,Ma:Survey}.
%When $\Po$ is pre-Lie algebra we find the following expression 
%for this map.

%HMK dette lemma er ikke  korrekt for postLie. Må enten korrigeres eller fjernes. Har ikke tid til å finne tilsvarende i postLie. 
%\begin{lemma}[\cite{Ma:Survey} Manchon survey]
%When $\Po$ is a pre-Lie algebra, the map $\Phi$ in (\ref{eq:PoPreVF}) is given
%as
%\[ \Phi(a) = a + \frac{1}{2}\, a \rhd a + \frac{1}{6} \, a \rhd (a \rhd a) 
%+ \cdots . \]
%\end{lemma}
%
%\begin{proof}
%Show this, and possibly extend this to post-Lie algebras.
%\end{proof}

%HMK skriver om litt siste del av avsnittet
For $a,a^\prime \in \Po_{field}$ let 
\[ a \plstar a^\prime = \log^*(\exp^*(a) * \exp^*(a^\prime)), \]
a product which is computed using the BCH formula for the
Lie algebra $\lbra, \rbra$. With this product $\Po_{field}$
becomes a pro-unipotent group. Transporting this product
to $\Po_{flow}$ using the bijection $\Phi$ in (\ref{eq:PoPreVF}), 
we get for $b, b^\prime \in \Po_{flow}$ a product
\[ b\, \sharp \, b^\prime 
= \log^{\bullet}(\exp^{\bullet}(b) * \exp^{\bullet}(b^\prime)), \]
the {\it composition} product for principal flows. 

\begin{example} We continue Example \ref{exa:PrePoFlow}
Let $g = \sum_{i \geq 0} g_i h^i$ be another power series of vector fields,
$\exp^*(hg)$ its flow series, and $\Psi_h^g : \RR^n \pil \RR^n$ 
the flow it determines.
Let $c$ be the principal part of $\exp^*(hg)$.
The composition of the flows $\Psi^g_h \circ \Psi^f_h$ is the flow 
sending $\phi$ to 
\[ \exp^*(hg) (\exp^*(hf) \phi) = (\exp^*(hg) * \exp^*(hf)) \phi. \]
 The principal part of the composed flow is 
\[ \log^\bullet(\exp^*(hg) * \exp^*(hf)) = \log^\bullet(\exp^\bullet(c)
* \exp^\bullet(b)) = c \, \sharp \, b, \]
the Butcher product of $c$ and $b$.
\end{example}

Denote by $\plcdot$ the product in $\Po_{flow}$ given by the
BCH-formula for the Lie bracket 
%HMK endret til enkel Lie bracket
$[,]$, 
\[ x\, \plcdot \, y := \log^{\bullet}(\exp^{\bullet}(x) \bullet \exp^{\bullet}(y)).  \]
\begin{proposition}  \label{pro:PoPrexy}
For $x,y$ in the post-Lie algebra $\Po_{flow} $
we have
%HMK klargjør med parenteser rekkefølgen av operasjonene
\[ x \, \sharp \, y = x \plcdot \left(\exp^{\bullet}(x) \rhd y\right). \]
\end{proposition}

\begin{proof}
%GF-Har satt inn bedre referanser
From~ \cite[Prop.3.3]{MK-L-EF:Env} 
the product $A * B = \sum_{\Delta(A)} A_{(1)}(A_{(2)} \rhd B)$.
Since $\exp^\bullet (x)$ is a group-like element 
it follows that: 
\[ \exp^\bullet(x) * \exp^\bullet(y) = \exp^\bullet(x)\bullet\left(\exp^\bullet(x)\rhd\exp^\bullet(y)\right). \]
By \cite[Prop.3.1]{MK-L-EF:Env} 
$A \rhd BC = \sum_{\Delta(A)} (A_{(1)} \rhd B)(A_{(2)} \rhd C)$
and so again using that $\exp^\bullet(x)$ is group-like  and the expansion
of $\exp^\bullet(y)$: 
\[ \exp^\bullet(x)\rhd\exp^\bullet(y) = \exp^\bullet\left(\exp^\bullet(x)\rhd y\right). \]
Hence
\[x \, \sharp \, y = \log^\bullet\left(\exp^\bullet(x)\bullet\left(\exp^\bullet(x)\rhd\exp^\bullet(y)\right)\right) =
\log^\bullet\left(\exp^\bullet(x)\bullet\exp^\bullet(\exp^\bullet(x)\rhd y)\right). 
\]
\end{proof}

In the pre-Lie case $[,]=0$, therefore $\plcdot = +$ and we obtain the formula
derived in~\cite{PLT-EFP}
%HMK rettet formel
\[  x \, \sharp \, y = x + \exp^{\bullet}(x) \rhd y.\]

\subsection{Substitution}
Let $\End_{\postlie}(P) = \Hom_{\postlie}(P,P)$ be the endomorphisms
of $P$ as a post-Lie algebra. (In the special case that 
$P$ is a pre-Lie algebra, this is simply the endomorphisms of 
$P$ as a pre-Lie algebra.)
It is a monoid, but not generally a vector space. It acts
on the the post-Lie algebra $P$. 

Since the action respects the brackets $[,],\, \lbra,\rbra \, $ and $\rhd$,
it also acts on the enveloping algebra $U(P)$ and its completion 
$\hU(P)$, and respects the products $*$ and $\bullet$. Hence the
exponential maps $\exp^*$ and $\exp^\bullet$ are equivariant for this
action, and similarly the logarithms $\log^*$ and $\log^\bullet$. 
So the formal flow map
\[ \Phi : P_{\field} \lpil P_{\flow} \]
is equivariant for the action. The action on $P_\flow$ (which is 
technically the same as the action on $P_\field$), is called {\it substitution}
and is usually studied in a more specific context, as we do in 
Section \ref{sec:Action}. An element $\phi \in \End_{\postlie}(P)$ comes
from sending a field $f$ to a perturbed field $f^\prime$, and one
then sees how this affects the exact flow or approximate flow maps given by numerical algorithms. %HMK

\part{The algebraic geometric setting}
In this part we have certain finiteness assumptions on the Lie algebras
and pre- and post-Lie algebras, and so may consider them and binary
operations on them in the setting of varieties. 
%In this part we assume that $\kr$ is a field, which we denote by $k$. 

\section{Affine varieties and group actions}

We assume the reader is familiar with basic algebraic
geometry of varieties and morphisms, like presented in 
\cite[Chap.1]{Ha} or \cite[Chap.1,5]{CLO}. 
We nevertheless briefly recall basic notions. A notable and not
so standard feature
is that we in the last subsection define infinite dimensional
varieties and morphisms between them.

\subsection{Basics on affine varieties} \label{subsec:VarBasics}
Let $\kk$ be a field and $S = \kk[x_1, \ldots, x_n]$ the polynomial ring.
The {\it affine $n$-space} is 
\[ \AA_\kk^n = \{ (a_1, \ldots, a_n) \, | \, a_i \in k \}. \] 
An ideal $I \sus S$ defines and {\it affine variety} in $\AA_\kk^n$:
\[ X = \cZ(I) = \{ p \in \AA_\kk^n \, | \, f(p) = 0, \text{ for } f \in I \}. \]
Given an affine variety $X \sus \AA_\kk^n$, its associated ideal
is 
\[ \cI(X)  = \{ f \in S \, | \, f(p) = 0, \text{ for } p \in X \}. \]
Note that if $X = \cZ(I)$ then $I \sus \cI(X)$, 
and $\cI(X)$ is the largest ideal defining the variety $X$.
The correspondence 
\[ \mbox{ideals in } \kk[x_1, \ldots, x_n]  \bihom{\cZ}{\cI}
\mbox{ subsets of }  \AA_\kk^n \]
is a Galois connection. Thus we get a one-to-one correspondence
\[ \mbox{ image of } \cI \xleftrightarrow{ 1 - 1} 
\underset{ = \mbox{ varieties in } \AA_\kk^n}{\mbox{ image of }\cZ}. 
\]

\begin{remark} When the field $\kk$ is algebraically closed, 
Hilbert's Nullstellensatz says that the image of $\cI$ is precisely
the radical ideals in the polynomial ring. In general however
the image of $\cI$ is only contained in the radical ideals.
\end{remark}

The {\it coordinate ring} of a variety $X$ is the ring 
$A(X) = \kk[x_1, \ldots, x_n]/\cI(X)$. 
A morphism of affine varieties $f : X \pil Y$ where $X \sus \AA_\kk^n$ 
and $Y \sus \AA_\kk^m$ is a 
a map sending a point $\bfa = (a_1, \ldots, a_n)$ to a point 
$(f_1(\bfa), \ldots, f_m(\bfa))$ where the $f_i$ are polynomials
in $S$. This gives rise to a homomorphism of coordinate rings
\begin{align*} f^\rf : A(Y) & \lpil A(X) \\
\overline{y_i} & \lpil f_i(\overline{\bfx}), \quad i = 1, \ldots, m
\end{align*}
In fact this is a one-one correspondence:
\[ \{ \text{morphisms } f:X \pil Y\} \quad \xleftrightarrow{ 1 - 1 } 
\quad \{ \text{algebra homomorphisms } f^\rf : A(Y) \pil A(X)\}.\]

The zero-dimensional affine space $\AA_\kk^0$
is simply a point, and its coordinate ring is $k$.
Therefore to give a point $p \in \AA_\kk^n$ is equivalent
to give an algebra homomorphism
$\kk[x_1, \ldots, x_n] \pil \kk$. 

\begin{remark} \label{rem:VarBasicsK}
We may replace $\kk$ by any commutative ring $\kkr$. The
affine space $\AA_\kkr^n$ is then $\kkr^n$. The coordinate ring of this
affine space is $\kkr[x_1, \ldots, x_n]$. A point $p \in \AA_\kkr^n$
still corresponds to an algebra homomorphism $\kkr[x_1, \ldots, x_n] \pil \kkr$.
Varieties in $\AA_\kkr^n$ may be defined in the same way, and there
is still a Galois connection between ideals in 
$\kkr[x_1, \ldots, x_n]$ and subsets of $\AA_\kkr^n$, and a
one-one correspondence between morphisms of varieties
and coordinate rings.
\end{remark}

\medskip
The affine space $\AA_\kk^n$ comes with the {\it Zariski topology}, whose
closed sets are the affine varieties in $\AA_\kk^n$ and whose open
sets are the complements of these. This induces also the Zariski topology
on any affine subvariety $X$ in $\AA_\kk^n$. 

\medskip 
If $X$ and $Y$ are affine varieties in $\AA^n_\kk$ and $\AA^m_\kk$ respectively, 
their product $X \times Y$ is 
an affine variety in $\AA^{n+m}_\kk$ whose ideal is the ideal in 
$\kk[x_1, \ldots,x_n, y_1, \ldots, y_m]$ generated by $\cI(X) + \cI(Y)$.
Its coordinate ring is 
\[ A(X \times Y) = A(X) \te_\kk A(Y). \]

\medskip
If $A$ is a ring and $f \neq 0$ in $A$, we have the localized ring 
$A_f$ whose elements are all $a/f^n$ where $a \in A$. Two
such elements $a/f^n$ and $b/f^m$ are equal if $f^k(f^ma - f^nb) = 0$
for some $k$. If $A$ is an integral domain, this is equivalent to
$f^ma - f^nb = 0$. Note that the localization $A_f$ is isomorphic to the
quotient ring $A[x]/(xf-1)$. Hence if $A$ is a finitely generated $k$-
algebra, $A_f$ is also a finitely generated $k$-algebra. A consequence of this
is the following: Let $X$ be an affine variety in $\AA_\kk^n$ whose ideal is 
$I = \cI(X)$ contained in $\kk[x_1, \ldots, x_n]$, 
and let $f$ be a polynomial function.
The open subset 
\[D(f) = \{ p \in X \, | \, f(p) \neq 0 \} \sus X \] is then 
in bijection to the variety
$X^\prime \in \AA_\kk^{n+1}$ defined by the ideal $I + (x_{n+1}f-1)$.
This bijection is actually a homeomorphism in the Zariski topology.
The coordinate ring 
\[ A(X^\prime) = A(X)[x_{n+1}]/(x_{n+1}f-1) \iso A(X)_f. \]
Hence we identify $A_f$ as the coordinate ring of $D(f)$ and
can consider $D(f)$ as an affine variety. Henceforth we shall drop the adjective
affine for a variety, since all our varieties will be affine.

\subsection{Coordinate free descriptions of varieties}
For flexibility of argument, it may desirable to consider varieties in
a coordinate free context. 

Let $V$ and $W$ be dual finite dimensional vector spaces. 
So $V = \Hom_\kk(W,\kk) = W^*$, and then
$W$ is naturally isomorphic to $V^* = (W^*)^*$. We consider $V$ as
an affine space (this means that we are forgetting the structure
of vector space on $V$). Its coordinate ring is the symmetric algebra
$\Sym(W)$. Note that any polynomial
$f \in \Sym(W)$ may be evaluated on any point $\bfv \in  V$, 
since $\bfv : W \pil \kk$ gives maps $\Sym_d(W) \pil \Sym_d(\kk) = \kk$ and
thereby a map $\Sym(W) = \oplus_d \Sym_d(W) \pil \kk$.   

Given an ideal $I$
in $\Sym(W)$, the associated affine variety is 
\[ X = \{ \bfv \in V \, | \,  f(\bfv) = 0, \text{ for } f \in I \} \sus V. \]
Given a variety $X \sus V$ we associate the ideal 
\[ \cI(X) = \{ f \in \Sym(W) \, | \, f(\bfv) = 0, \text{ for } \bfv \in X \} 
\sus \Sym(W). \]
The coordinate ring of $X$ is $A(X) = \Sym(W)/\cI(X)$.

\medskip
Let $W^1$ and $W^2$ be two vector spaces, with 
dual spaces $V^1$ and $V^2$. 
A map $f : X^1 \pil X^2$ between varieties in these spaces is a 
map which is given by polynomials once a coordinate system is fixed
for $V^1$ and $V^2$. Such a map then gives a 
a homomorphism of coordinate rings
$f^\rf: \Sym(W^2)/I(X^2) \pil \Sym(W^1)/I(X^1)$,
and this gives a one-one correspondence between morphisms 
$f$ between $X^1$ and $X^2$ and 
algebra homomorphisms $f^\rf$ between their coordinate rings.

\subsection{Affine spaces and monoid actions}
The vector space of linear operators on $V$ is denoted $\End(V)$.
It is an affine space 
with $\End(V) \iso \AA_\kk^{n \times n}$, and with coordinate ring
$\Sym(\End(V)^*)$. 
%\iso \kk[t_{ij}]_{i,j = 1,\ldots,n}. 
We then have an action 
\begin{align} \label{eq:AffineEV} \End(V) \times V & \pil V \\ \notag
(\phi,v) & \mapsto \phi(v).
\end{align}
This is a morphism of varieties. Explicitly, if $V$ has basis 
$e_1, \ldots, e_n$ an element in $\End(V)$ may be represented by a
matrix $A$ and the map is given by:
\[ (A,(v_1, \ldots, v_n)^t) \mapsto A \cdot (v_1, \ldots, v_n)^t, \]
which is given by polynomials.

%The coordinate ring $A(V)$ of $V$ is the symmetric algebra
%$\Sym(V^*) \iso \kk[x_1, \ldots, x_n]$ where $V = \Hom(V,\kr)$ is the dual
%vector space of $V$ and $x_1, \ldots, x_n$ is the dual basis of $e_1, \ldots,
%e_n$, the coordinate functions on $V$ with respect to the basis of $e$'s.

%The coordinate ring of $\End(V)$ is $\Sym(\End(V)^*) \iso \kk[t_{ij}]$ where 
%$i,j = 1, \ldots, n$. 
The morphism of varieties (\ref{eq:AffineEV})
then corresponds to the algebra homomorphism on coordinate rings
\[ \Sym(V^*) \pil \Sym(\End(V)^*) \te_\kk \Sym(V^*). \]
With a basis for $V$, the coordinate ring $\Sym(\End(V)^*)$ is
isomorphic to the polynomial  ring $\kk[t_{ij}]_{i,j = 1, \ldots,n}$,
where the $t_{ij}$ are coordinate functions on $\End(V)$, and
the coordinate ring $\Sym(V^*)$ is isomorphic to
$k [x_1, \ldots, x_n]$ where the $x_i$ are coordinate functions on
$V$. The map above on coordinate rings is then given by
\[ x_i \mapsto \sum_{j} t_{ij}x_j. \]

We may also consider the set $\GL(V) \sus \End(V)$ of invertible
linear operators. This is the open subset $D(\det(\phi))$ of 
$\End(V)$ defined by the nonvanishing of the determinant. 
Hence, fixing a basis of $V$, its coordinate ring
is the localized ring $\kk[t_{ij}]_{\det((t_{ij}))}$, by the
last part of Subsection \ref{subsec:VarBasics}.  The set $\SL(V) \sus \End(V)$
are the linear operators with determinant $1$. This is a closed
subset of $\End(V)$ defined by the polynomial equation $\det((t_{ij})) -1 = 0$.
Hence the coordinate ring of $\SL(V)$ is the quotient ring
$\kk[t_{ij}]/(\det((t_{ij})) -1)$.

\medskip
Now given an affine monoid variety $M$, that is an affine variety with 
a product morphism $\mu : M \times M \pil M$ which is associative.
Then we get an algebra homomorphism of coordinate rings
\[ A(M) \mto{\Delta} A(M) \te_\kk A(M).\]
Since the following diagram commutes
\[ \begin{CD} M \times M \times M @>{\mu \times \bfe}>> M \times M \\
@V{\bfe \times \mu}VV @VV{\mu}V \\
M \times M @>{\mu}>> M,
\end{CD} \] 
we get a commutative diagram of coordinate rings:
\[ \begin{CD} A(M) \te_\kk A(M) \te_\kk A(M) @<<{\Delta \te \bfe}< A(M) 
\te A(M) \\
@A{\bfe \te \Delta}AA @AA{\Delta}A \\
A(M) \te_\kk A(M) @<{\Delta}<< A(M).
\end{CD} \] 
%Thus the algebra $A(M)$ with $\Delta$ becomes
%a commutative bialgebra. 
The zero-dimensional affine space $\AA_\kk^0$
is simply a point, and its coordinate ring is $k$.
A character on $A(M)$ is an algebra homomorphism $A(M) \pil \kk$.
On varieties this gives a morphism $P = \AA_\kk^0 \pil M$, or 
a point in the monoid variety. In particular the unit in $M$
corresponds to a character $A(M) \mto{\epsilon} k$, the counit. Thus the
algebra $A(M)$ with $\Delta$ and $\epsilon$ becomes a bialgebra.

%The counit $A(M) \mto{\eta} \kr$ corresponds to the identity element
%$e \in M$. Let $m = \ker \eta$. An infinitesimal character on $A(M)$
%is a linear map $m/m^2 \pil \kr$, and this is precisely an element of the
%Zariski tangent space $T_e(M)$ of $M$ at $e$.

\medskip
%GF-har gjort noen endringer her fram til neste seksjon
%som gir den generelle situasjonen, og så den mer spesielle
The monoid may act on a variety $X$ via a morphism of varieties
\begin{equation} \label{eq:AffineMV} M \times X \pil X. 
\end{equation}
On coordinate rings we get a homomorphism of algebras,
\begin{equation} \label{eq:AffineCoM} A(X) \pil A(M) \te_\kk A(X),
\end{equation}
making $A(X)$ into a comodule algebra over the bialgebra
$A(M)$.

In coordinate systems the morphism (\ref{eq:AffineMV}) may
be written: 
\[ (m_1, \ldots, m_r) \times (x_1, \ldots, x_n) \mapsto
(f_1(\bfm, \bfx), f_2(\bfm,\bfx), \ldots ) . \]
%It gives a morphism of coordinate rings
%\begin{equation} \label{eq:AffineCoM} \Sym(W) \pil A(M) \te_\kk \Sym(W). 
%\end{equation}
If $X$ is an affine space $V$ and the action comes from a morphism
of monoid varieties $M \pil \End(V)$, the action by $M$ is linear
on $V$. Then
$f_i(\bfm, \bfv) = \sum_j f_{ij}(\bfm) v_j$.
The homomorphism on coordinate rings (recall that $V = W^*$)
\[ \Sym(W) \pil A(M) \te_\kk \Sym(W)\]
is then induced from a morphism
\begin{align*} 
W & \pil A(M) \te_\kk W  \\
x_j & \mapsto \sum_i f_{ij}(\mathbf{u}) \te_\kk x_i
\end{align*}
where the $x_j$'s are the coordinate functions on $V$ and $\mathbf{u}$
are the coordinate functions on $M$. 

\medskip
We can also consider an affine group variety $G$ with a morphism
$G \pil GL(V)$ and get a group action 
$G \times V \pil V$. The inverse morphism for the group, induces
an antipode on the coordinate ring $A(G)$ making it
a commutative Hopf algebra.

\subsection{Infinite dimensional affine varieties and monoid actions}
\label{subsec:Infinite}

The infinite dimensional affine space $\AA^\infty_\kk$ is
$\prod_{i \geq 1} \kk$. Its elements are infinite sequences
$(a_1, a_2, \ldots)$ where the $a_i$ are in $k$. 
Its coordinate ring is the polynomial ring in infinitely many
variables $S = \kk[x_i, i \in \NN]$. 

An ideal $I$ in $S$, defines an affine variety 
\[ X = V(I) = \{ \bfa \in \AA^\infty_\kk \, | \, f(\bfa)  = 0, \text{ for } 
f \in I \}.\]
Note that a polynomial $f$ in $S$ always involves only a finite
number of the variables, so the evaluation $f(\bfa)$ is meaningful.
Given an affine variety $X$, let its ideal be:
\[ \cI(X) = \{ f \in S \, | \, f(\bfa) = 0 \text{ for } \bfa \in X \}. \]
The coordinate ring $A(X)$ of $X$ is the quotient ring 
$S/\cI(X)$. The affine subvarieties of $\AA_\kk^\infty$ form
the closed subsets in 
the Zariski topology on $\AA_\kk^\infty$, and this then induces
the Zariski topology on any subvariety of $\AA_\kk^\infty$.

\medskip
A morphism $f: X \pil Y$ of two varieties, is a map such that
$f(\bfa) = (f_1(\bfa), f_2(\bfa), \ldots )$ where each $f_i$
is a polynomial function (and so involves only a finite number of the 
coordinates of $\bfa$)

Letting $\kk[y_i, i \in \NN]$ be the coordinate ring of affine space where
$Y$ lives, we get a morphism of coordinate rings
\begin{align*}
f^\rf : A(Y) & \pil A(X) \\
        \overline{y_i} & \mapsto f_i(\overline{\bfx})
\end{align*}
This gives a one-one correspondence
\[ \{ \text{morphisms } f:X \pil Y\}  \quad \leftrightarrow 
\quad \{ \text{algebra homomorphisms } f^\rf : A(Y) \pil A(X) \}.\]

\medskip
For flexibility of argument, it is desirable to have a coordinate free
definition of these varieties also. The following includes then both the
finite and infinite-dimensional case in a coordinate free way.

Let $W$ be a vector space with a countable basis. We get the
symmetric algebra $\Sym(W)$. Let $V = \Hom_\kk(W,\kk)$ be the dual 
vector space, which will be our affine space.
% Note that any polynomial
%$f \in Sym(W)$ may be evaluated on any point $\bfv \in  V$, 
%since $\bfv : W \pil \kr$ gives maps $\Sym_d(W) \pil \Sym_d(\kr) = \kr$ and
%thereby a map $\Sym(W) = \oplus_d \Sym_d(W) \pil \kr$.   
Given an ideal $I$
in $\Sym(W)$, the associated affine variety is 
\[ X = V(I) = \{ \bfv \in V \, | \,  f(\bfv) = 0, \text{ for } f \in I \}. \]
Given a variety $X$ we associate the ideal 
\[ \cI(X) = \{ f \in \Sym(W) \, | \, f(\bfv) = 0, \text{ for } \bfv \in X \}. \]
Its coordinate ring is $A(X) = \Sym(W)/\cI(X)$.
%If $W_\fin \sus W$ is a finite dimensional subspace, we get
%a dual projection map $V \pil V_\fin$. Let $X_\fin \sus V_\fin$ be the image of% 
%$X$ in this subspace,  and let $\overline{X_\fin} \sus V_\fin$ be its closure in the
%Zariski topology. It is the affine variety whose ideal in $\Sym(W_\fin)$ is
%$I(X_\fin) = \Sym(W_\fin) \cap \cI(X)$. 
We shall shortly define morphism between varieties. In order for these to
be given by polynomial maps, we will need filtrations
on our vector spaces. Given a filtration by finite dimensional
vector spaces
\[ \langle 0 \rangle = W_0 \sus W_1 \sus W_2 \sus \cdots \sus W.\]
On the dual space $V$ we get a decreasing filtration by
$V^i = \ker ( (W)^* \pil (W_{i-1})^*)$.
The affine variety $V/V^i \iso (W_{i-1})^*$ has coordinate ring
$\Sym(W_{i-1})$.
If $X$ is a variety in $V$ its image $X_i$ in the finite affine space
$V/V^i$ need not be Zariski closed. Let $\overline{X_i}$ be its closure.
This is an affine variety in $V/V^i$ whose ideal is 
$\cI(X) \cap \Sym(W_{i-1})$.

\medskip
%Let $W^1$ and $W^2$ be two vector spaces with countable bases,
%and with filtrations of finite-dimensional subspaces.
%Let $V_1$ and $V_2$ be their dual affine spaces.
%\[ W^1_1 \sus W^1_2 \sus \cdots \sus W^1, \quad W^2_1 \sus W^2_2 \sus 
%\cdots \sus W^2. \] 
%We get dual spaces $V_1 = (W^1)^*$ and $V_2 = (W^2)^*$, and a decreasing
%filtration
%$V_1^i = \ker { (W^1)^* \pil (W^1_i)^*$, and similarly $V_2^i \sus V_2$.
%We will now define morphisms between 
%varieties. To make this correspond to morphisms between coordinate rings
%we must assume some finiteness condition. 
A map $f : X_1 \pil X_2$ between varieties in these spaces is a {\it morphism
of varieties} if there are decreasing filtrations 
\[ V_1 = V_1^1 \supseteq V_1^2 \supseteq \cdots, \quad V_2 = V_2^1 \supseteq
V_2^2 \supseteq \cdots\]
with finite dimensional quotient spaces, 
such that for any $i$ we have a commutative diagram
%finite dimensional subspace $W^2_\fin \sus W^2$, there
%is a finite dimensional subspace $W^1_\fin \sus W^1$ such that the 
%map $f$ factors by a commutative diagram
\[ \begin{CD} X_1 @>f>> X_2 \\
            @VVV @VVV \\
            \overline{X_{1,i}} @>>> \overline{X_{2,i}}
\end{CD} \]
and the lower map is a morphism between varieties
in $V_1/V_1^i$ and $V_2/V_2^i$.

We then get a homomorphisms of coordinate rings
\begin{equation} \label{eq:VarFprimi} 
f^{\rf}_i : \Sym(W^2_i)/\cI(X_{2,i}) \pil \Sym(W^1_i)/\cI(X_{1,i}), 
\end{equation}
and the direct limit of these gives a homomorphism of coordinate rings
\begin{equation} \label{eq:VarFprim}
f^{\rf} : \Sym(W^2)/\cI(X_2) \pil \Sym(W^1)/\cI(X_1). 
\end{equation}
Conversely given an algebra homomorphism $f^\sharp$ above. Let
\[  W_1^2 \sus W_2^2 \sus W_3^2 \sus \cdots \]
be a filtration. Write $W^1 = \oplus_{i \in \NN} kw_i$ in terms
of a basis. The image of $W_i^2$ will involve only a finite number
of the $w_i$. Let $W_i^1$ be the f.d. subvector space generated by these $w_i$.
Then we get maps \eqref{eq:VarFprimi}, giving morphisms
\[ \begin{CD} \overline{X_{1,i+1}} @>>> \overline{X_{2,i+1}} \\
            @VVV @VVV \\
            \overline{X_{1,i}} @>>> \overline{X_{2,i}}.
\end{CD} \]
In the limit we then get a morphism of varieties $f: X_1 \pil X_2$.
This gives a one-one correspondence between morphisms of
varieties $f: X_1 \pil X_2$ and algebra homomorphisms $f^\rf$.

\medskip
If $X^1$ and $X^2$ be varieties in the affine spaces $V^1$ and $V^2$. 
Their product $X^1 \times X^2$ is
a variety in the affine space $V^1 \times V^2$ which is the dual space
of $W^1 \oplus W^2$. Its  coordinate ring is 
$A(X^1 \times X^2) = A(X^1) \te_\kk A(X^2)$. 

If $M$ is an affine monoid variety (possibly infinite dimensional) 
its coordinate ring $A(M)$ becomes a commutative bialgebra. 
If $M$ is an affine group variety, then $A(M)$ is a Hopf algebra.
We can again further consider
an action on the affine space
\[ M \times V \pil V. \]
It correspond to a homomorphism of coordinate rings
\[ \Sym(W) \pil A(M) \te_\kk \Sym(W),\] making
$\Sym(W)$ into a comodule algebra over $A(M)$.
If the action by $M$ is linear on $V$, the algebra homomorphism
above is induced by a linear map
$W \pil A(M) \te_\kk W$. 

\section{Filtered algebras with finite dimensional  quotients}

In this section we assume the quotients $L_p = L/L^{p+1}$ are finite
dimensional vector spaces. This enables us to define the
dual Hopf algebra $U^c(K)$ of the enveloping algebra $U(L)$. 
This Hopf algebra naturally identifies as the coordinate ring
of the completed Lie algebra $\hat{L}$. In Subsection \ref{subsec:FinFilBCH}
the Baker-Campbell-Haussdorff product on the variety $L$ is shown to 
correspond to the natural coproduct on the dual Hopf algebra $U^c(K)$. In 
the last Subsection \ref{subsec:FinFilButcher}
the Lie-Butcher product on a post-Lie algebra is also shown to correspond
to the natural coproduct on the dual Hopf algebra.

\subsection{Filtered Lie algebras with finite dimensional quotients}
Recall that $\fg_p$ is the quotient $\fg/\fg^{p+1}$. 
The setting in this
section is $\kk$ is a field of characteristic zero, and 
that these quotients $\fg_p$ are finite dimensional as $\kk$-vector
spaces. We assume that $L$ is complete with respect to this cofiltration, so
we have the inverse limit
\[ L = \hat{L} = \underset{p} \varprojlim L_p. \]

The dual $\fc^p = \Hom_\kk(\fg_p,\kk)$ is a finite dimensional Lie
coalgebra. Let $\fc = \underset{p}{\varinjlim}\, \fc^p$ be the direct limit.
Recall that the quotient algebra 
\[ U^j(\fg_p) = U(\fg_p)/F^{j+1}U(\fg_p). \]
The dual $U^j(\fg_p)^*$ is a finite dimensional coalgebra $U^c_{j}(\fc^p)$, 
and we have inclusions
\[ \begin{CD} U^c_j(\fc^p) @>{\sus}>> U^c_{j+1}(\fc^p) \\
@V{{\sus}}VV @VV{\sus}V \\
U^c_j(\fc^{p+1}) @>{\sus}>> U^c_{j+1}(K^{p+1}). 
\end{CD} \]
We have the direct limits 
\[ U^c(\fc^p) := \underset{j}{\varinjlim}\,  U^c_j(\fc^p), 
\quad U^c(\fc) := \underset{j,p}{\varinjlim} U^c_j(\fc^p). \]

\begin{lemma} 
$U^c(K)$ is a sub Hopf algebra of $T^c(K)$ with the shuffle product.
\end{lemma}

\begin{proof}
$U^j(L_p)$ is a quotient algebra of $T(L_P)$ and $T(L)$, and so 
$U_j^c(K_p)$ is a subcoalgebra of $T^c(K_p)$ and $T^c(K)$. 
The coproduct on $U(L_p)$, the shuffle coproduct,
does not descend to a coproduct on $U^j(L_p)$.
But we have a well defined co-map
\[ U^{2j-1}(L_p) \pil U^j(L_p) \te U^j(L_p) \]
compatible with the shuffle coproduct on $T(L_p)$. 
Dualizing this we get 
\[ U^c_j(K_p) \te U^c_j(K_p) \pil U^c_{2j-1}(K_p) \]
and taking colimits, we get $U^c(K)$ as a subalgebra of $T^c(K)$
with respect to the shuffle product.
\end{proof}

\begin{proposition} \label{pro:FinFilDual} There are isomorphisms
\begin{itemize}
\item[a.] $\fg \iso  \Hom_\kk(\fc,\kk)$ of Lie algebras,
\item[b.] $\hU(\fg) \iso \Hom_\kk(U^c(\fc),\kk)$ of algebras.
\item[c.] The coproduct on $U^c(K)$ is dual to the completed product
on $\hU(L)$
\[ U^c(K) \mto{\Delta_\bullet} U^c(K) \te U^c(K), \quad
\hU(L) \hte \hU(L) \mto{\bullet} \hU(L). \]
\end{itemize}
\end{proposition}

\begin{proof}
a. Since $\fg$ is the completion of the $\fg^p$, 
it is clear that there is a map of Lie algebras 
$\Hom_\kk(\fc,\kk) \pil \fg$. We need only show that this is an isomorphism
of vector spaces.

It is a general fact that for any object $N$ in a category
$\cC$ and any indexed diagram $F : J \pil \cC$ then
\[ \Hom(\varinjlim F, N) \iso \varprojlim \Hom(F(-),N). \]
Applying this to the category of $\kk$-vector spaces enriched in 
$\kk$-vector spaces (meaning that the Hom-sets are $\kk$-vector spaces),
we get
\[ \Hom_\kk(K,\kk) = \Hom_{k}(\varinjlim K^p,L) 
= \varprojlim \Hom(K^p,\kk) = \varprojlim L^p = \hat{L}.\]
b. This follows as in b. above.

c. This follows again by the above. Since tensor products
commute with colimits we have
\[ U^c(K) \te U^c(K) = \underset{p,j}\varinjlim U^c_j(K^p) \te U^c_j(K^p). \]
Then 
\begin{align*}  \Hom_\kk(U^c(K) \te U^c(K), \kk) = &  \Hom_\kk(\varinjlim 
U^c_j(K^p) \te U^c_j(K^p), \kk) \\
 = & \underset{p,j}\varprojlim U^j(L^p) \te U^j(L^p)
= \hU(L) \hte \hU(L). 
\end{align*}

% So let $V \pil \fg_p$ be compatible maps.
%Dualizing we get compatible maps $\fc^p \pil V^*$. Hence we get a
%unique map $\fc \pil V^*$. Dualizing we get maps
%\[ V \pil V^{**} \pil \Hom_\kk(\fc,\kk) \] compatible with the $V \pil \fg_p$.
%But this means $\Hom_\kk(\fc,\kk)$ is the colimit of the $\fg_p$ as
%vector spaces. 

%b. The argument for part b. is similar, using Lemma \ref{Lem:ExpLogLimjp}.
\end{proof}

\medskip
The coalgebra $U^c(\fc)$ is a Hopf algebra with the shuffle
product. It has unit $\eta$ and counit $\epsilon$. 
Denote by $\star$ the convolution product on this Hopf algebra, and
by $\ben$ the identity map. Write $\ben = \eta \circ \epsilon + J$. The Euler
idempotent 
\[ e : U^c(\fc) \pil U^c(\fc) \]
is the convolution logarithm
\[ e = \log^{\star}(\ben) = \log^{\star}(\eta \circ \epsilon + J) 
= J  - 
J^{\star 2}/2 + J^{\star 3}/3
- \cdots . \]
%This map is called the Euler idempotent and denoted by $e$.

\begin{proposition} \label{pro:FinFilUtilK}
The image of $U^c(\fc) \mto{e} U^c(\fc)$ is $\fc$. This inclusion
of $\fc \sus U^c(\fc)$ is a section of the natural map $U^c(\fc) \pil \fc$. 
\end{proposition}

\begin{proof}
This follows the same argument as Proposition 
\ref{pro:ExpLogEuler}.
\end{proof}

This gives a map 
$\fc \pil U^c(\fc)$. Since $U^c(\fc)$ is a commutative algebra under
the shuffle product, we get a map from the free commutative algebra
$\Sym(\fc) \pil U^c(\fc)$.

\begin{proposition} \label{pro:FinFilPBW}
This map 
\begin{equation} \label{eq:FinFilPsi} 
\psi : \Sym(\fc) \mto{\iso} U^c(\fc) 
\end{equation}
is an isomorphism of commutative algebras.
\end{proposition}

\begin{proof}
By Proposition \ref{pro:ExpLogPBW} there is an isomorphism of coalgebras
\[ U(L_p) \mto{\iso} \Sym^c(L_p) \]
and the filtrations on these coalgebras correspond. Hence
we get an isomorphism
\[ U^j(L_p) \mto{\iso} \Sym^{c,j}(L_p). \]
Dualizing this we get
\[ \Sym_j(K^p) \mto{\iso} U^c_j(K^p). \]
Taking the colimits of this we get the statement.
\end{proof}

\medskip
In $\Hom_\kk(U^c(\fc),\kk)$ there are two distinguished subsets.
The {\it characters} are the algebra homomorphisms 
$\Hom_{Alg}(U^c(\fc),\kk)$. Via the isomorphism of Proposition 
\ref{pro:FinFilDual}
it corresponds to the grouplike elements of $\hU(\fg)$.
The {\it infinitesimal characters}
are the linear maps $\alpha : U^c(\fc) \pil \kk$ such that 
\[ \alpha(uv) = \epsilon(u) \alpha(v) + \alpha(u) \epsilon(v). \]
We denote these as $\Hom_{Inf}(U^c(\fc),\kk)$. 

\begin{lemma} \label{lem:FinFilChar}
Via the isomorphism in Proposition \ref{pro:FinFilDual}b.
These characters correspond naturally to the following:
\begin{itemize}
\item[a.] $\Hom_{Inf}(U^c(\fc),\kk) \iso \Hom_\kk(\fc,\kk) \iso \fg$.
\item[b.] $\Hom_{Alg}(U^c(\fc),\kk) \iso G(\hU(\fg))$.
\end{itemize}
\end{lemma}

\begin{proof}
a. The map $U^c(K) \mto{\phi} K$ from Proposition \ref{pro:FinFilUtilK}
has kernel $\kk \oplus U^c(K)_+^{\shu 2}$. We then see that any linear
map $K \pil \kk$ induces by composition an infinitesimal character on
$U^c(K)$. Conversely given an infinitesimal character $\alpha : U^c(K) \pil \kk$
then both $\kk$ and $U^c(K)_+^{\shu 2}$ are seen to be in the kernel,
and so such a map is induced from a linear map $K \pil \kk$ by composition
with $\phi$. 

b. That $s :U^c(K) \pil \kk$ is an algebra homomorphism is equivalent
to the commutativity of the diagram
\begin{equation} \label{eq:FinFilUKAlg}
 \xymatrix{  U^c(K) \te U^c(K) \ar[r] \ar[d]^{s \te s} & U^c(K) \ar[d]^{s} \\
   \kk \te_\kk \kk \ar[r] & \kk }.
\end{equation}
But this means that by the map 
\begin{align*} \hU(L) & \pil \hU(L) \hte \hU(L)  \\
s & \mapsto s \te s.
\end{align*}
Conversely given a grouplike element $s \in \hU(L)$, it corresponds
by Proposition \ref{pro:FinFilDual} b. to $s : U^c(K) \pil \kk$, and it being
grouplike means precisely that the diagram \eqref{eq:FinFilUKAlg} commutes.
\end{proof}

On $\Hom_\kk(U^c(\fc),\kk)$ we also have the convolution product, 
which we again denote by $\star$. Note that
by the isomorphism in Proposition \ref{pro:FinFilDual}, this corresponds
to the product on $\hU(L)$. 
Let $\Hom_\kk(U^c(\fc),\kk)_+$ consist of the $\alpha$ with $\alpha(1) = 0$.
We then get the exponential map
\[ \Hom_\kk(U^c(\fc),\kk)_+ \mto{\exp} \epsilon + \Hom_\kk(U^c(\fc),\kk)_+ \]
given by
\[ \exp(\alpha) = \epsilon + \alpha + \alpha^{\star 2}/2! + \alpha^{\star 3}/3!
+ \cdots .\] 
This is well defined since $U^c(\fc)$ is a conilpotent coalgebra
and $\alpha(1) = 0$.
Correspondingly we get 
\[ \epsilon + \Hom_\kk(U^c(\fc),\kk)_+ \mto{\log} \Hom_\kk(U^c(\fc),\kk)_+ \]
given by
\[ \log(\epsilon+ \alpha) = \alpha - \frac{\alpha^{\star 2}}{2} + 
\frac{\alpha^{\star 3}}{3}
- \cdots . \]

\begin{lemma} \label{lem:FinFilExpLog} The maps
\[ \Hom_\kk(U^c(K), \kk)_+ \bihom{\exp}{\log} \epsilon + \Hom_\kk(U^c(K),\kk)_+ \]
give inverse bijections. They restrict to the inverse bijections
\[ \Hom_{Inf}(U^c(\fc),\kk) \bihom{\exp}{\log} \Hom_{Alg}(U^c(\fc),\kk). \]
\end{lemma}

\begin{proof}
Using the identification of Proposition \ref{pro:FinFilDual} the 
$\exp$ and $\log$ maps above correspond to the 
$\exp$ and $\log$ maps in Proposition \ref{pro:ExpLogExpLog}.
\end{proof}

Since $\Sym(\fc)$ is the free symmetric algebra on $\fc$,
there is a bijection
$\Hom_{Alg}(\Sym(\fc),\kk) \mto{\iso} \Hom_\kk(\fc,\kk)$.
The following shows that all the various maps
correspond.

\begin{proposition}
The following diagram commutes, showing that the various horizontal
bijections correspond
to each other:
\[ \begin{CD}
\Hom_\kk(\fc,\kk) @>{\iso}>> \Hom_{Alg}(\Sym(\fc),\kk) \\
@|  @AA{\psi^*}A \\
\Hom_\kk(\fc,\kk) @>{\exp}>> \Hom_{Alg}(U^c(\fc),\kk) \\
@V{\iso}VV @VV{\iso}V \\
\fg @>{\exp}>> G(\hU(\fg))
\end{CD} \]
\end{proposition}

\begin{proof}
That the lower diagram commutes is clear by the proof of 
Lemma \ref{lem:FinFilExpLog}. The middle (resp. top) map sends $K \pil \kk$ 
to the unique algebra homomorphism $\phi$ (resp. $\phi^\prime$) 
such that the following 
diagrams commute
\[ \xymatrix{ K \ar[r] \ar[rd] & U^c(K) \ar[d]^{\phi} \\
                  & \kk}, \quad
 \xymatrix{ K \ar[r] \ar[rd] & \Sym(K) \ar[d]^{\phi} \\
                  & \kk}. 
\]
%Similarly the upper map sends $K \pil \kk$ to the unique algebra
%homomorphism such that the following diagram commutes
%\[ \xymatrix{ K \ar[r] \ar[rd] & \Sym(K) \ar[d]^{\phi} \\
%                  & \kk}. 
%\]
Since the following diagram commutes where $\psi$ is the
isomorphism of algebras
\[ \xymatrix{ K \ar[r] \ar[rd] & \Sym(K) \ar[d]^{\psi} \\
                  & U^c(K)}, 
\]
the commutativity of the upper diagram in the statement of the proposition
follows.
\end{proof}

\subsection{Actions of endomorphisms}

Let $E = \End_{Lie \,\, co}(K)$ be the endomorphisms of $K$ as a Lie
co-algebra, which also {\it respect the filtration} on $K$. 

\begin{proposition}
The Euler map in Proposition \ref{pro:FinFilUtilK} 
is equivariant for the endomorphism action.
Hence the isomorphism $\Psi : \Sym(K) \pil U^c(K)$ is equivariant for
the action of the endomorphism group $E$.
\end{proposition}

\begin{proof}
The coproduct on $U^c(K)$ is clearly equivariant for $E$ and similarly
the product on $U^c(K)$ is equivariant, since $U^c(K)$ is a subalgebra of 
$T^c(K)$ for the shuffle product. Then if $f,g : U^c(K) \pil U^c(K)$
are two equivariant maps, their convolution product $f \star g$ is also
equivariant. 

 Since $\ben$ and $\eta \circ \epsilon$ are equivariant for $E$, the
difference $J = \ben - \eta \circ \epsilon$ is so also.
The Euler map $e = J - J^{\star 2}/2 + J^{\star 3}/3 - \cdots$ must
then be equivariant for the action of $E$.

 Since the image of the Euler map is $K$, the inclusion 
$K \hookrightarrow U^c(K)$ is equivariant also, and so the map $\Psi$ above.
\end{proof}

As a consequence of this the action  of $E$ on $K$ induces an action 
on the dual Lie algebra $L$ respecting its filtration.
By Proposition \ref{pro:FinFilDual} this again induces a diagram of actions of the following 
{\it sets}
\begin{equation} \label{eq:FinFilEL}
\begin{CD}  E \times \hU(L) @>>> \hU(L) \\
@VVV @VVV \\
E \times L @>>> L.
\end{CD}
\end{equation}

\subsubsection{The free Lie algebra} \label{sssec:FinFilLie}
Now let $V = \oplus_{i \geq 1} V_i$ be a positively 
graded vector spaces with finite
dimensional parts $V_i$. We consider the special case of the above
that $L$ is the completion $\hLie(V)$ of the free Lie algebra on $V$. 
Note that $\Lie(V)$ is a graded 
Lie algebra with finite dimensional graded parts. 
The enveloping algebra $U(\Lie(V))$ is the tensor algebra $T(V)$.

The graded dual vector space is $V^\gd = \oplus V_i^*$ and the graded
dual Lie co-algebra is $\Lie(V)^\gd$. The Hopf algebra
$U^c(\Lie(V)^\gd)$ is the shuffle Hopf algebra $T(V^\gd)$. 

Since $\Lie(V)$ is the free Lie algebra on $V$, 
the endomorphisms $E$ identifies as (note that here it is
essential that we consider endomorphisms respecting the filtration) 
\begin{equation} \label{eq:FinFilEndco}
 \End_{\Lie \, co}(\Lie(V)^\gd, \Lie(V)^\gd) ) = \Hom_{\Lie}(\Lie(V), \hLie(V)).
\end{equation}
%\[ \Hom_{\Lie}(\Lie(V), \widehat{\Lie}(V)) 
%\iso \Hom_\kk(V, \widehat{\Lie}(V)). \]
This is a variety with coordinate ring $\cE_V = \Sym( V \te \Lie(V)^\gd)$,
which is a bialgebra.
Furthermore the diagram \eqref{eq:FinFilEL} with $L = \hLie(V)$ in this
case will be a morphism of varieties: Both $E,L$ and $\hU(L)$ come with
filtrations and all maps are given by polynomial maps. So we get a dual
diagram of coordinate rings
\[ \begin{CD}
\Sym(\Lie(V)^\gd) @>>> \cE_V \te \Sym(\Lie(V)^\gd) \\
@VVV @VVV \\
\Sym(T(V)^\gd) @>>> \cE_V \te \Sym(T(V)^\gd)
\end{CD}. \]
But since the action of $E$ is linear on $\hLie(V)$ and $\hat{T}(V)$, 
this gives a diagram
\[ \begin{CD}
\Lie(V)^\gd @>>> \cE_V \te \Lie(V)^\gd \\
@VVV @VVV \\
T^c(V^\gd) @>>> \cE_V \te T^c(V^\gd)
\end{CD}, \]
and so the isomorphism $\Sym(\Lie(V)^\gd) \mto{\iso} T^c(V^\gd)$
is an isomorphism of comodules over the algebra $\cE_V$.

\subsection{Baker-Campbell-Haussdorff on coordinate rings}
\label{subsec:FinFilBCH}
The space $K$ has a countable basis and so we may consider
$\Sym(K)$ as the coordinate ring of the variety $L = \Hom_\kk(K,\kk)$. 
By the isomorphism $\psi:\Sym(K) \mto{\iso} U^c(K)$ of 
Proposition \ref{pro:FinFilPBW} we may think
of $U^c(K)$ as this coordinate ring. Then also 
$U^c(K) \te_\kk U^c(K)$ is the coordinate ring of $L \times L$.

The coproduct (whose dual is the
product on $\hU(L)$) 
\[ U^c(K) \mto{\Delta_\bullet} U^c(K) \te_\kk U^c(K), \]
will then correspond to a morphism of varieties
$L \times L \pil L$. The following explains what it is.

\begin{proposition}
The map $L \times L \pil L$ given by
\[ (a,b) \mapsto \log^\bullet (\exp^\bullet(a) \bullet \exp^\bullet(b)) \]
is a morphism of varieties, and on coordinate rings it corresponds
to the coproduct
\[ U^c(K) \mto{\Delta_\bullet} U^c(K) \te U^c(K). \]
\end{proposition}

This above product on $L$ is the Baker-Campbell-Haussdorff product.

\begin{example}
Let $V = \oplus_{i \geq 1} V_i$ be a graded vector space with finite
dimensional graded parts. Let $\Lie(V)$ be the free Lie algebra
on $V$, which comes with a natural grading. The enveloping algebra
$U(\Lie(V))$ is the tensor algebra $T(V)$. The dual Lie coalgebra
is the graded dual $K = \Lie(V)^{\gd}$, and $U^c(K)$ is the
graded dual tensor coalgebra $T(V^{\gd})$ which comes with the shuffle product. 
Thus the shuffle algebra $T(V^\gd)$ identifies as the {\it coordinate ring}
of the Lie series, the completion $\widehat{\Lie}(V)$ 
of the free Lie algebra on $V$.

The coproduct on 
$T(V^{\gd})$ is the deconcatenation coproduct. This can then be
considered as an extremely simple codification of the
Baker-Campbell-Haussdorff formula for Lie series in the
completion $\widehat{\Lie}(V)$. 
\end{example}

\begin{proof}
If $X \pil Y$ is a morphism of varieties and $A(Y) \mto{\phi} A(X)$ the 
corresponding homomorphism of coordinate rings, then the point
$p$ in $X$ corresponding to the algebra homomorphism $A(X) \mto{p^*} \kk$
maps to the point $q$ in $Y$ corresponding to the algebra homomorphism
$A(Y) \mto{q^*} \kk$ given by $q^* = \phi \circ p^*$. 

Now given a points $a$ and $b$ in $L = \Hom_\kk(K,\kk)$. They correspond to
algebra homomorphisms from the coordinate ring $U^c(K) \mto{\tilde{a}, 
\tilde{b}} \kk$, the unique such extending $a$ and $b$, and these
are $\tilde{a} = \exp(a)$ and $\tilde{b} = \exp(b)$. 
The pair $(a,b) \in L \times L$ corresponds to the
homomorphism on coordinate rings
\[ \exp(a) \te \exp(b) : U^c(K) \te U^c(K) \mto{\tilde{a} \te \tilde{b}} 
\kk \te_\kk \kk = \kk.\]
Now via the coproduct, which is the homomorphism of coordinate
rings,
\[ U^c(K) \mto{\Delta_\bullet} U^c(K) \te U^c(K) \]
this maps to the algebra homomorphism 
$\exp(a) \bullet \exp(b) : U^c(K) \pil \kk$. This is the algebra 
homomorphism corresponding to the following point in $L$: 
\[ \log^\bullet(\exp(a) \bullet \exp(b)) : K \pil \kk. \]
\end{proof}

\subsection{Filtered pre- and post-Lie 
algebras with finite dimensional quotients}
\label{subsec:FinFilButcher}
We now assume that the filtered quotients $\Po/\Po^p$, which
again are post-Lie algebras, are all 
finite dimensional.
Let their duals be $\CPo_p = \Hom_\kk(\Po/\Po^p,\kk)$ and
$\CPo = \underset{p} \varinjlim \, \CPo_p$, which is a post-Lie coalgebra.
We shall assume $\Po = \hat{\Po}$ is complete with respect to this 
filtration.
Then $\Po = \Hom(\CPo,\kk)$, and $\Sym(\CPo)$ is the coordinate ring
of $\Po$. There are two Lie algebra structures on $P$, given
by $[,]$ and $\lbra, \rbra$ of Definition \ref{def:PrePoRel}. These
correspond to the products $\bullet$ and $*$ on the enveloping
algebra of $P$. We shall use the first product $\bullet$, giving
the coproduct $\Delta_\bullet$ on $U^c(Q)$. For this coproduct
Proposition \ref{pro:FinFilPBW} gives an isomorphism
\begin{equation} \label{eq:PoPreCPO}
\psi_\bullet : \Sym(\CPo) \mto{\iso} U^c(\CPo).
\end{equation}
Due to the formula in Proposition  
\ref{pro:PoPrexy}
the product
\[ \Po \times \Po \mto \sharp \Po \]
on each quotient $P/P^i$, is given by 
polynomial expressions. It thus corresponds to a 
homomorphism of coordinate rings 
\begin{equation} \label{eq:PoPreDeltaSharp}
\Sym(\CPo) \mto{\Delta_{\sharp}} \Sym(\CPo) \te \Sym(\CPo).
\end{equation}

\begin{proposition} \label{pro:FinFilBDual}
Via the isomorphism $\psi_\bullet$ in (\ref{eq:PoPreCPO})
the coproduct $\Delta_\sharp$ above corresponds to the
coproduct
\[ U^c(\CPo) \mto{\Delta_*} U^c(\CPo) \te U^c(\CPo), \]
which is the dual of the product $*$ on $U(\Po)$.
\end{proposition}

%GF-satt inn to nye bemerkninger
\begin{remark}
In order to identify the homomorphism of coordinate rings as the
coproduct $\Delta_*$ it is essential that one uses the isomorphism
$\psi_\bullet$ of \eqref{eq:PoPreCPO}. If one uses another isomorphism
$\Sym(Q) \mto{\iso} U^c(Q)$ like the isomorphism $\psi_*$ derived
from the coproduct $\Delta_*$, the statement is not correct.
See also the end of the last remark below.
\end{remark}

\begin{remark}{\it The Connes-Kreimer Hopf algebra.}
For the free pre-Lie algebra $T_C$ 
(see the next Section \ref{sec:Free}) this
identifies the Connes-Kreimer Hopf algebra $\cH_{CK}$ as the {\it coordinate
ring $\Sym(T_C^\gd)$} of the Butcher series $\hat{T}_C$ under the
Butcher product.

As a variety the Butcher series $\hat{T}_C$ is endowed with the
Zariski topology, and the Butcher product is continuous for this topology.
In \cite{Bo-Sc} another finer topology on $\hat{T}_C$ is considered
when the field $\kk = \RR$ or $\CC$.
\end{remark}

\begin{remark} {\it The MKW Hopf algebra.}
For the free post-Lie algebra $P_C$ (see Section \ref{sec:Free}) it 
identifies the MKW Hopf algebra
$T(\OT_C^\gd)$ as the {\it  coordinate ring $\Sym(\Lie(\OT_C)^\gd)$} of the
Lie-Butcher series $\hat{P}_C = \widehat{\Lie}(\OT_C)$. 
A (principal) Lie-Butcher series 
$\ell \in \hat{P}_C$ corresponds to an
element $\Lie(\OT_C)^\gd \mto{\ell} k$. This lifts via the isomorphism
$\psi_\bullet$ of \eqref{eq:PoPreCPO} to a character of the shuffle algebra 
$T(\OT_C^\gd) \mto{\tilde{\ell}} k$. That the lifting from (principal) LB series
to character of the MKW Hopf algebra must be done
using the inclusion $\Lie(\OT_C)^\gd \inpil T(\OT_C^\gd)$ via
the Euler map of Proposition \ref{pro:FinFilUtilK} associated to the
coproduct $\Delta_\bullet$, is a technical
point which has not been made explicit previously.
\end{remark}

\begin{proof}[Proof of Proposition \ref{pro:FinFilBDual}.]
Given points $a,b \in P$. They correspond to linear maps $Q \mto{a,b} k$.
Via the isomorphism $\psi_\bullet$ these extend to algebra homomorphisms
$U^c(Q) \mto{\tilde{a}, \tilde{b}} k$, where $\tilde{a} = \exp^\bullet(a)$
and $\tilde{b} = \exp^{\bullet}(b)$. The pair $(a,b) \in P \times P$ 
then corresponds to a homomorphism of coordinate rings
\[ \exp^\bullet(a) \te \exp^\bullet(b) : 
U^c(Q) \te U^c(Q) \mto{\tilde{a} \te \tilde{b}} 
\kk \te_\kk \kk = \kk.\]
Now via the coproduct associated to $*$, 
which is the homomorphism of coordinate
rings,
\[ U^c(Q) \mto{\Delta_*} U^c(Q) \te U^c(Q) \]
this maps to the algebra homomorphism 
$\exp^\bullet(a) * \exp^\bullet(b) : U^c(Q) \pil \kk$. This is the algebra 
homomorphism corresponding to the following point in $P$: 
\[ \log^\bullet(\exp^\bullet(a) * \exp^\bullet(b)) : Q \pil \kk. \]
\end{proof}

\section{Free pre- and post-Lie algebras}
\label{sec:Free}
This section recalls free pre- and post-Lie algebras, and the
notion of substitution in these algebras.

\subsection{Free post-Lie algebras}
We consider the set of rooted planar trees, or ordered trees: 
\[ \OT = \{ \ab, \aabb, \aababb, \aaabbb, \aabababb, \aabaabbb, \aaabbabb,\cdots \}, \]
and let $\kk \OT$ be the $\kk$-vector space with these
trees as basis.
It comes with an operation $\rhd$, called {\it grafting}. For two trees $t$ and
$s$ we define $t \rhd s$ to be the sum of all trees obtained by 
attaching the root of $t$ with a new edge onto a vertex of $s$, 
with this new edge as the leftmost branch into $s$.

If $C$ is a set, we can color the vertices of $\OT$ with the elements of 
$C$. We then get the set $\OT_C$ of labelled planar trees.
The {\it free post-Lie algebra} on $C$ is the free Lie algebra 
$P_C = \Lie(\OT_C)$ on the set of $C$-labelled planar trees. The grafting
operation is extended to the free Lie algebra $\Lie(\OT_C)$ by using
the relations \ref{def:PrePoRel}. Note that $P_C$ has a natural 
grading by letting $P_{C,d}$ be the subspace generated by all
bracketed expressions of trees with a total number of $d$ leaves.
In particular $P_C$ is filtered.

The enveloping algebra of
$P_C$ identifies as the tensor algebra $T(\OT_C)$.
It was introduced and studied in \cite{MK-W}, see 
also \cite{MK-F} for more on the computational aspect in this algebra.
 Its completion
identifies as 
\[ \hat{T}(\OT_C)  = \prod_{d \geq 0} T(\OT_C)_d. \]

\medskip
\subsection{Free pre-Lie algebras}
Here we consider instead (non-ordered) rooted trees
\[ T =\{ \ab, \aabb, \aababb, \aaabbb, \aabababb, \aabaabbb=\aaabbabb,\cdots \}. \]
On the vector space $\kk T$ we can similarly define grafting $\rhd$.
Given a set $C$ we get the set $T_C$ of trees labelled by $C$.
The free pre-Lie algebra is $\kk T_C$, \cite{ChaLiv}. Its enveloping algebra
is the symmetric algebra $\Sym(T_C)$, called the Grossmann-Larsen algebra,
and comes with the ordinary symmetric product $\cdot$ and the product $*$,
\cite{OudGuin}.

\subsection{Substitution}
We consider free post-Lie algebras $P_C$. Let $M$ be a manifold
and $\cX M$ the vector fields on this manifold. In the setting
of Lie-Butcher series, $\cX M$ can be endowed with the structure of 
a post-Lie algebra, see \cite{MK-L:Formal}. 
Let $f \in \cX M$
be a vector field, and $P_\bullet$ the free post-Lie algebra
on one generator. We get a map $\bullet \mapsto f$ and so a
map of post-Lie algebras $P_\bullet \pil \cX M$, which
sends a tree $t$ to the associated {\it elementary differential} of $f$.
If $f \in \cX M \lbra h \rbra$ we similarly get a map of post-Lie
algebras $P_\bullet \pil \cX M \lbra h \rbra$,

We get a commutative diagram of flow maps
\[ \begin{CD}
  P_{\bullet,\field} @>{\Phi_P}>> P_{\bullet,\flow} \\
  @VVV  @VVV \\
  \cX M \lbra h \rbra_{\field} @>{\Phi_{\cX M}}>> 
  \cX M \lbra h \rbra_{\flow}. 
  \end{CD}
\]
The field $f$ is mapped to the flow $\Phi_{\cX M}(f)$. By perturbing the
vector field $f \pil f + \delta$, it is sent to a flow 
$\Phi_{\cX M}(f + \delta)$.
We assume the perturbation  $\delta$ is expressed in terms of the
elementary differentials of $f$, and so it comes from a 
perturbation $ \bullet \pil \bullet + \delta^\prime = s$. 
Since $\Hom(\bullet, P_\bullet) = \End_{\postlie}(P_\bullet)$ this 
gives an endomorphism of the post-Lie algebra. We are now interested in
the effect of this endomorphism on the flow, called {\it substitution}
of the perturbed vector field, and we are interested
in the algebraic aspects of this action. We study this for the
free post-Lie algebra $P_C$, but most of the discussions below are of
a general nature, and applies equally well to the free pre-Lie algebra, 
and generalizes the results of \cite{EF:TwoInteract}.

\section{Action of the endomorphism group
and substitution in free post-Lie algebras}
\label{sec:Action}

Substitution in the free pre-Lie or free post-Lie algebras
on one generator give, by dualizing, the operation of co-substitution in
their coordinate rings, which are the Connes-Kreimer and the
MKW Hopf algebras. 
%%GF
In \cite{EF:TwoInteract} they show that co-substitution
on the Connes-Kreimer algebra is governed by a bialgebra $\cH$ such
that the Connes-Kreimer algebra $\cH_{CK}$ is a comodule bialgebra over this
bialgebra $\cH$. Moreover $\cH_{CK}$ and $\cH$ are isomorphic as
commutative algebras. This is the notion of two bialgebras in 
{\it cointeraction}, a situation further studied in \cite{Ma-OrG},
\cite{Fo-Erh}, and \cite{Fo-Chr}.

   In this section we do the analog for the MKW Hopf algebra, and in
a more general setting, since we consider free pre- and 
post-Lie algebras on any finite
number of generators. In this case $\cH_{CK}$ and $\cH$ are no longer
isomorphic as commutative algebras. %%
As we shall see the situation is understood very well by using the
algebraic geometric setting and considering the MKW Hopf algebra as the
coordinate ring of the free post-Lie algebra. The main results
of \cite{EF:TwoInteract} also follow, and are understood better,
by the approach we develop here.

\subsection{A bialgebra of endomorphisms}
Let $C$ be a finite dimensional vector space over the field $\kk$, 
and $P_C$ the 
free post-Lie algebra on this vector space. 
 It is a graded vector space 
$P_C = \bigoplus_{d \geq 1} P_{C,d}$ graded by the number
of vertices in bracketed expressions of trees, and so has finite dimensional
graded pieces.
It has a graded dual 
\[ P_C^\gd = \oplus_d \Hom_\kk(P_{C,d},\kk). \]
Let $\{ l \}$ be a basis for $P_C$. It gives a dual
basis $\{ l^*\}$ for $P_C^\gd$.
The dual of $P_C^\gd$ is the completion 
\[ \hP_C = \Hom_\kk(P_C^\gd, \kk) = \underset{d}{\varprojlim} P_{C, \leq d}. \]
It is naturally a post-Lie algebra and comes with a decreasing
filtration $\hP_C^{d+1} = \ker (\hP_C \pil P_{C,\leq d}$.   

\medskip
Due to the freeness of 
$P_C$ we have:
\[ \Hom_\kk(C, P_C) = \Hom_{\postlie}(P_C,P_C) = \End_{\postlie}(P_C). \]
Denote the above vector space as $E_C$. 
If we let $\{c\}$ be a basis for $C$, 
the graded dual $E_C^\gd = C \te_\kk P_C^\gd$ has a basis 
$\{ a_c(l) := c \te l^*\}$. 

The dual of $E_C^\gd$ is $\hE_C = \Hom_\kk(E_C^\gd,\kk)$ which may be written
as $C^* \te_\kk \hP_C$. This is an affine 
space with coordinate ring 
\[ \cE_C := \Sym(E_C^\gd) = \Sym(\Hom_\kk(C,P_C)^\gd) = \Sym( C \te_\kk P_C^\gd). \]
The filtration on $\hP_C$ induces also a filtration on $\hE_C$.

\medskip
A map of post-Lie algebras $\phi:P_C \pil \hat{P}_C$ induces a map 
of post-Lie algebras $\hat{\phi}: \hat{P_C} \pil \hat{P}_C$. 
We then get the inclusion
\[ \hE_C = \Hom_{\postlie}(P_C, \hat{P}_C) \sus \Hom_{\postlie}(\hat{P}_C,
\hat{P}_C). \]
If $\phi,\psi \in \hE_C$, we get a composition $\psi \circ \hat{\phi}$,
which we by abuse of notation write as $\psi \circ \phi$. 
This makes
$\hE_C$ into a monoid of affine varieties:
\[ \hE_C \times \hE_C \mto{\circ} \hE_C. \]
It induces a homomorphism on coordinate rings:
\[ \cE_C \mto{\Delta_\circ} \cE_C \te \cE_C. \]
This coproduct is coassociative, since $\circ$ on $\hE_C$ is
associative. Thus $\cE_C$ becomes a bialgebra.

Note that when $C = \langle \bullet \rangle$ is one-dimensional, then
\[ \cE_{\bullet} = \Sym(P_{\bullet}^\gd) \iso T_{\shuffle}(\OT_{\bullet}^\gd) \]
as {\it algebras}, using Proposition \ref{pro:FinFilPBW}. 
The coproduct $\Delta_\circ$ considered on the shuffle algebra is, however, 
neither deconcatenation nor the Grossmann-Larsen coproduct.
A concrete description of this coproduct is given in \cite{EF:TwoInteract}.

\subsubsection{Hopf algebras of endomorphisms}
The augmentation map $P_C \pil C$ gives maps
\[ \Hom_\kk(C,P_C) \pil \Hom_\kk(C,C)\] and dually
\[ \Hom_\kk(C,C)^\gd \sus \Hom_\kk(C,P_C)^\gd \iso C \te_\kk P_C^\gd.\]
Recall that $a_c(d)$ are the basis elements of $\Hom_\kk(C,C)^\gd$
(the coordinate functions on $\Hom_\kk(C,C)$), where
$c$ and $d$ range over a basis for $C$. We can then invert
$D = \det(a_c(d))$ in the coordinate ring $\cE_C$. This gives
a Hopf algebra $\cE_C^{\times}$ which is the localized ring
$(\cE_C)_D$. Another possibility is to divide $\cE_C$
by the ideal generated by $D-1$. This gives a Hopf algebra 
$\cE_C^{1} = \cE_C/(D-1)$. 
A third possibility is to to divide $\cE_C$ out
by the ideal generated by the $a_c(d) - \delta_{c,d}$. This 
gives a Hopf algebra $\cE_C^{\text{Id}}$. In the case $C = \{ \bullet \}$
both the latter cases give the Hopf algebra $\overline{\cH}$
in \cite{EF:TwoInteract}.

\subsection{The action on the free post-Lie algebra}
\label{subsec:Action}
The monoid $E_C$ acts on $P_C$, and $\hE_C$ acts on $\hP_C$. So we 
get a morphism of affine varieties
\begin{equation} \label{eq:SubStar} \hE_C \times \hP_C \mto{\star} \hP_C 
\end{equation}
called {\it substitution}.

Let $\cH_C = \Sym(P_C^\gd)$ be the coordinate ring
of $\hP_C$. We get a homomorphism of coordinate rings
called {\it co-substitution}
\begin{equation} \label{eq:SubStarKord}\cH_C \mto{\Delta_\star} \cE_C \te \cH_C.
\end{equation}
Note that the map in (\ref{eq:SubStar}) is linear in the second factor
so the algebra homomorphism  (\ref{eq:SubStarKord}) comes from a linear map
\[ P_C^\gd \pil \cE_C \te P_C^\gd. \]
The action $\star$ gives a commutative diagram
\[ \begin{CD} 
\hE_C \times \hE_C \times \hP_C 
@>{1 \times \star }>> \hE_C \times \hP_C \\
@V{\circ \times \ben}VV @VV{\star}V \\
\hE_C \times \hP_C 
@>{\star}>> \hP_C 
\end{CD} \]
which dually gives a diagram
\[ \begin{CD} 
\cE_C \te \cE_C \te \cH_C
@<<< \cE_C \te \cH_C \\
@AAA @AAA \\
\cE_C \te \cH_C 
@<<< \cH_C.
\end{CD} \]

This makes $\cH_C$ into a comodule over $\cE_C$, in fact a comodule
algebra, since all maps are homomorphisms of algebras.
The Butcher product $\rf$ on $\hP_C$ is dual to the coproduct
$\Delta_\star : \cH_C \pil \cH_C \te \cH_C$ by Proposition 
\ref{pro:FinFilBDual}. Since $\hE_C$ gives an endomorphism
of post-Lie algebra we have for $a \in \hE_C$ and $u,v \in \hP_C$:
%Since the elements of $\hE_C$ are homomorphisms of post-Lie algebras,
%we will for $a \in \hE_C$, inducing $\hP_C \mto{a\star} \hP_C$, have 
\[a\star(u \rf v) = (a\star u) \rf (a\star v). \]
In diagrams
\[ \begin{CD} 
\hE_C \times \hE_C \times \hP_C \times \hP_C
@>{\ben \times \tau \times \ben}>> 
\hE_C \times \hP_C \times \hE_C \times \hP_C
@>{ \star \times \star}>> \hP_C \times \hP_C
 \\
@A{{\rm{diag}} \times \ben \times \ben }AA 
@.  @VV{\rf}V\\
\hE_C \times \hP_C \times \hP_C @>{\ben \times \rf}>>
\hE_C \times \hP_C @>{\star}>> \hP_C 
\end{CD} \]
which dually gives a diagram
\[ \begin{CD} 
\cE_C \te \cE_C \te \cH_C \te \cH_C @<{\ben \te \tau \te \ben}<<
\cE_C \te \cH_C \te \cE_C \te \cH_C @<{\Delta_\star \te \Delta_\star}<<
 \cH_C \te \cH_C \\
@VVV @. @AA{\Delta_*}A \\
\cE_C \te \cH_C \te \cH_C @<{\ben \te \Delta_*}<< 
\cE_C \te \cH_C @<{\Delta_\star}<< \cH_C.
\end{CD} \]
This makes $\cH_C$ into a comodule Hopf algebra over $\cE_C$.
We also have 
\[a\star(u \gft v) = (a\star u) \gft (a\star v) \]
giving corresponding commutative diagrams, making $\cH_C$ into
a comodule algebra over $\cE_C$.

%*(Must elaborate some more here, do it also for the $*$ product, and maybe
 %only for this)*

\subsubsection{The identification with the tensor algebra}
\label{subsubsec:SymTe}

The tensor algebra $T(\OT_C)$ is the enveloping algebra of $P_C = \Lie(OT_C)$. 
The endomorphism of post-Lie co-algebras 
$\End_{\postlie \text{-co}}(P_C^\gd)$ identifies by equation
\eqref{eq:FinFilEndco} as $\hE_C =
\Hom_{\postlie}(C, \hat{P}_C)$. It is an endomorphism submonoid of 
$\End_{\Lie}(P_C^\gd)$

By Subsubsection \ref{sssec:FinFilLie} the isomorphism 
$\cH_C = \Sym(P_C^\gd) \mto{\iso} T^c(\OT_C^\gd)$ is equivariant for the action of $\hE_C$ and
induces a commutative diagram
\begin{equation} \label{eq:SubHTId}
\begin{CD}
\cH_C @>{\Delta_\star}>> \cE_C \te \cH_C \\
@V{\iso}VV @VV{\iso}V \\
T^c(\OT_C^\gd) @>{\Delta_{T,\star}}>> \cE_C \te T^c(\OT_C^\gd)
\end{CD} 
\end{equation}
Thus all the statements above in Subsection \ref{subsec:Action} may be
phrased with $T^c(\OT_C^\gd)$ instead of $\cH_C$ as comodule over $\cE_C$.
%We will use this in the sequel. 
%*Probably some more and better arguments in this subsection*

%it has an action by the monoid $E_C = \End_{\postlie}(P_C)$.
%It extends to an action on completions
%\begin{equation} \label{eq:SubETT} \hE_C \times \hat{T}(\OT_C) \pil 
%\hat{T}(\OT_C) 
%\end{equation}
%and so induces a homomorphism on coordinate rings
%\[ \Sym(T(\OT_C^\gd)) \pil \Sym(E_C^\gd) \te \Sym(T(\OT_C^\gd)). \]
%Since the action (\ref{eq:SubETT}) is linear on the second factor, this
%comes from a linear map
%\[ T(\OT_C^\gd) \mto{\Delta^\prime} \Sym(E_C^\gd) \te T(\OT_C^\gd). \]

%By Proposition \ref{pro:FinFilPBW} there is an isomorphism of commutative
%algebras
%\[ T(\OT_C^\gd) \iso \Sym(\Lie(\OT_C^\gd)) = \Sym(P_C^\gd). \]
%The monoid $E_C = \End_{\postlie}(P_C)$ acts on both the above and by 
%Section \ref{sec:PrePo} the isomorphism is equivariant for this action
%This gives a commutative diagram
%\[ \begin{CD}
%\cH_C @>{\Delta_\star}>> \cE_C \te \cH_C \\
%@V{\iso}VV @VV{\iso}V \\
%T(\OT_C^\gd) @>{\Delta^\prime_\star}>> \cE_C \te T(\OT_C^\gd)
%\end{CD} \]
%Thus all the statments in Subsection \ref{subsec:Action} may be
%phrased with $T(\OT_C)$ as a comodule over $\cE_C$. 
%*Probably some more and better arguments in this subsection*

\subsection{The universal substitution}
Let $K$ be a commutative $\kk$-algebra. We then get
$P_{C,K}^\gd = K \te_\kk P_C^\gd$, and correspondingly we get
\[ E_{C,K}^\gd, \quad  \cH_{C,K} = \Sym(P_{C,K}^\gd), \quad \cE_{C,K} =
\Sym(E_{C,K}^\gd). \]
Let the completion $\hat{P}_{C,K} = \Hom(P_{C,K}^\gd,K)$. (Note that
this is not $K \te_\kk \hat{P}_C$ but rather larger than this.)
Similarly we get $\hE_{C,K}$.
The homomorphism of coordinate rings
$\cH_{C,K} \pil \cE_{C,K} \te_K \cH_{C,K}$ corresponds to a map of affine
$K$-varieties (see Remark \ref{rem:VarBasicsK})
\begin{equation} \label{eq:SubStarK} \hE_{C,K} \times \hP_{C,K} \pil \hP_{C,K}.
\end{equation}
A $K$-point $A$ in the affine variety $\hE_{C,K}$ then corresponds
to an algebra homomorphism $\cE_{C,K} \mto{A^*} K$, and $K$-points
$p \in \hP_{C,K}$ corresponds to algebra homomorphisms $\cH_{C,K} \mto{p^*} K$.

In particular the map obtained from (\ref{eq:SubStarK}), using
$A \in \hE_{C,K}$:
\begin{equation} \label{eq:SubAstar} \hP_{C,K} \mto{A_\star} \hP_{C,K} 
\end{equation}
corresponds to the morphism on coordinate rings
\begin{align} \label{eq:SubHCK-Koord}
 \cH_{C,K} \pil \cH_{C,K} \te_K \cE_{C,K} \mto{\ben \te A^*}
\cH_{C,K} \te_K K = \cH_{C,K} 
\end{align}
which due to (\ref{eq:SubAstar}) being linear, comes from a $K$-linear map
\[ P^\gd_{C,K} \pil P^\gd_{C,K}. \]

\medskip
Now we let $K$ be the commutative algebra $\cE_C = \Sym(E_C^\gd)$. 
Then 
\[ \cE_{C,K} = K \te_\kk \Sym(E_C^\gd) = \Sym(E_C^\gd) \te \Sym(E_C^\gd). \] 
There is a canonical algebra homomorphism
\begin{equation} \label{eq:SubEK} \cE_{C,K} \mto{\mu} K 
\end{equation}
which is simply the product
\[ \Sym(E_C^\gd) \te_\kk \Sym(E_C^\gd) \mto{\mu}  \Sym(E_C^\gd). \]

\begin{definition}
Corresponding to the algebra homomorphism $\mu$ of \eqref{eq:SubEK} 
is the point $U$
in $\hE_{C,K} = \Hom_\kk(C_K, \hP_{C,K})$. 
This is the {\it universal} map (here we use the completed tensor product):
\begin{equation} \label{eq:EndUmap}
C \pil C \te (P_C^\gd \hat{\te} P_C) 
\end{equation} sending 
%(recall that $\{l\}$ is a basis for $P_C$)
\[ c  \mapsto c \te \underset{\scriptstyle \begin{matrix} l \text{ basis} \\
\text{element of } P_C \end{matrix}} \sum l^* \te l
= \sum_l a_c(l) \te l \]
Using this, (\ref{eq:SubAstar}) becomes the {\it universal} substitution, 
the $K$-linear map
\[ \hP_{C,K} \mto{U_\star} \hP_{C,K}. \]
Let $H = \Hom(C,P_C)^\gd$, the degree one part of $K = \cE_C$, and
$P_{C,H} = H \te_\kk P_C$. Note that the universal map \eqref{eq:EndUmap}
is a map from $C$ to $\hat{P}_{C,H}$.
\end{definition}

If $a \in \hE_C$ is a particular endomorphism, it corresponds to an 
algebra homomorphism (character) 
\begin{align*}  K = \cE_C & \mto{\alpha} \kk  \\ 
a_c(l) = c \te l^* & \mapsto \alpha(c \te l^*).
\end{align*}
Then $U_\star$ induces the substitution
$\hP_C \mto{a \star} \hP_C$
by sending each coefficient $a_c(l) \in K$ to 
$\alpha(c \te l^*) \in \kk$.

\medskip
The co-substitution $\cH_C \mto{\Delta_\star} \cE_C \te \cH_C$
of \eqref{eq:SubStarKord} induces a homomorphism
\[ \cE_C \te \cH_C \pil \cE_C \te \cE_C \te \cH_C \pil \cE_C \te \cH_C \]
which is seen to coincide with the homomorphism \eqref{eq:SubHCK-Koord}
when $K = \cE_C$. The universal substitution therefore corresponds to the map
on coordinate rings which is the co-substitution map, suitably lifted.

\medskip
Recall that the
tensor algebra $T(\OT_C)$ identifies as the forests of ordered trees
$\OF_C$. We may then write $T^c(\OT_C^\gd) = \OF_C^\gd$. By the diagram
\eqref{eq:SubHTId} the co-substitution $\cH_{C,K} \mto{\Delta_\star} \cH_{C,K}$ 
identifies as a map $\OF_{C,K}^\gd \mto{U_\star^T} \OF_{C,K}^\gd$ and we get 
a commutative diagram and its dual
\[ \xymatrix{\OF_{C,K}^\gd \ar[r]^{U_\star^T} \ar[d] & \OF_{C,K}^\gd \ar[d]\\
  P_{C,K}^\gd \ar[r]^{U_\star} & P_{C,K}^\gd}, \quad
\xymatrix{\hat{P}_{C,K} \ar[r]^{U_\star} \ar[d] & \hat{P}_{C,K} \ar[d]\\
\hat{\OF}_{C,K} \ar[r]^{U_\star} & \hat{\OF}_{C,K}}.
\]
We may restrict this to ordered trees and get
\[ \OF_{C,K}^\gd \mto{\oU^T_\star} \OT_{C,K}^\gd, \quad 
\hat{\OT}_{C,K} \mto{\oU_\star} \hat{\OF}_{C,K}. \]
We may also restrict and get 
\[ C_K \pil \hat{P}_{C,K} \pil \hat{\OF}_{C,K}, \]
with dual map
\begin{align} \label{eq:SubUt} U^t: \OF_{C,K}^\gd \pil P_{C,K}^\gd \pil C_K^*
\end{align} 
For use in Subsubsection \ref{sssec:OneGen}, note that \eqref{eq:EndUmap}
sends $C$ to $\hat{P}_{C,H}$ where $H = \Hom(C,P_C)^\gd \sus K$. 
A consequence is that $\OF_C^\gd \sus \OF_{C,K}^\gd$ is mapped
to $C_H^* \sus C_K^*$ by $U^t$.

%The universal substitution on ordered forests and its dual are
%\begin{equation} \label{eq:SubOFsub} \hat{\OF}_{C,K} \mto{U_\star} 
%\hat{\OF}_{C,K}, \quad \OF^\gd_{C,K} \mto{U^T_\star} \OF^\gd_{C,K}. 
%\end{equation}
%We may restrict this to ordered trees and get
%\[ \hat{\OT}_{C,K} \mto{\oU_\star} \hat{\OF}_{C,K}, 
%\quad \OF_{C,K}^\gd \mto{\oU^T_\star} \OT_{C,K}^\gd. \]
%
%Restricting the universal substitution (\ref{eq:SubOFsub})
%we get the universal substitution  on the free post-Lie algebra
%\[ U_\star : \hP_{C,K} \pil \hP_{C,K}. \]
%Restricting further we have
%\[ C_K \mto{U_{\star |\, C_K}} \hP_{C,K} \mto{\iota} \hat{T}_K(\OT_C) \]
%where $U_{\star | \, C_K}$ is the universal map \eqref{eq:EndUmap} used
%in the substituion. The composition $\iota \circ U_{\star | \, C_k}$ is dual to
%\[ U^t : T_K(OT_C^\gd) \pil P^\gd_{C,K} \pil C^*_K. \]

%The universal homomorphism $\cE_{C,K} \pil K$ comes from 
%a $K$-linear map $C_K \te_K P_{C,K}^\gd \pil Kwhich gives a map
%$P_{C,K}^\gd \pil C^*_K$is gives rise to 
%\[ U^t : T_K(OT_C) \pil P_{C,K} \pil C^*_K. \]
%Dually this is the universal map which we use in substitution
%\begin{align*}
%U : C_K & \pil \hP_{C,K} \inpil \hat{T}_K(\OT_C) \\
%c & \mapsto c \te \sum_l l^* \hat{\te} l = \sum_{l} a_c(l) \hat{\te} \l
%\end{align*}
%which is the restriction of the universal substitution
%\[U_\star: \hP_{C,K} \pil \hP_{C,K} \inpil \hat{T}_K(\OT_C). \]

\subsection{Recursion formula}
The universal substitution is described in \cite{MK-L:Back}, and we recall
it.
%Recall that the
%tensor algebra $T(\OT_C)$ identifies as the forests of ordered trees
%$\OF_C$. 
By attaching the trees in a forest to a root $c \in C$, 
there is a  natural isomorphism
\[\OT_C \iso \OF_C \te C \]
and dually 
\begin{equation} \label{eq:SubOFcOT} \OF_C^\gd \te C^* \mto{\iso} \OT_C^\gd 
%\omega \te \rho & \mapsto \omega \pode \rho
\end{equation}
Here we denote the image of $\omega \te \rho$ as $\omega \pode \rho$.

\begin{proposition}[\cite{MK-L:Back}] The following gives a partial recursion
formula for $\oU^T_\star$, the universal co-substitution followed by 
the projection onto the dual ordered trees:
\[\oU^T_\star(\omega) = \sum_{\Delta_{\gft}(\omega)} U^T_\star(\omega^{(1)}) \pode
U^t (w^{(2)}).\]
\end{proposition}

\begin{proof}
Recall the following general fact. Two maps $V \mto{\phi} W$ and 
$W^* \mto{\psi} V^*$ are dual iff for all $v \in V$ and $w^* \in W^*$ 
the pairings
\[ \langle v, \psi(w^*) \rangle = \langle \phi(v), w^* \rangle. \]
We apply this to 
$\phi = \overline{U_\star}$ and 
\[ \psi: \hat{\OF}_{C,K}^\gd \mto{\Delta_\gft} \OF_{C,K}^\gd \hat{\te} \OF_{C,K}^\gd
\mto{U_\star^T \te U^t} \OF_{C,K}^\gd \te C_K^* \mto{\pode} \OT_{C,K}^\gd. \] 

We must then show that
\[ \underset{\Delta_\gft(\omega)}\sum \langle t, 
U^T_\star(\omega^{(1)}) \pode U^t(w^{(2)})) \rangle 
=  \langle \oU_\star(t), \omega \rangle \]
So let $t = f \gft c$. Using first the above fact on the map
(\ref{eq:SubOFcOT}) and its dual: 
\begin{align*} 
\underset{\Delta_\gft(\omega)}\sum \langle t, U^T_\star(\omega^{(1)}) \pode U^t(w^{(2)}) \rangle 
=  & \underset{\Delta_\gft(\omega)}\sum \langle f \te c,  
U^T_\star (\omega^{(1)}) \te 
U^t(\omega^{(2)}) \rangle \\
 = & \underset{\Delta_\gft(\omega)}\sum \langle f, U^T_\star(\omega^{(1)}) 
\rangle \cdot \langle c, U^t(\omega^{(2)}) \rangle \\
 = & \underset{\Delta_\gft(\omega)}\sum 
\langle U_\star(f), \omega^{(1)} \rangle \cdot 
\langle U(c), \omega^{(2)} \rangle \\
= & \langle U_\star(f) \te U(c), \Delta_\gft(\omega) \rangle \\
= & \langle U_\star(f) \gft U(c), \omega \rangle \\
= & \langle \oU_\star(f \gft c), \omega \rangle = \langle \oU_\star(t), 
\omega \rangle
\end{align*}
\end{proof}

We now get the general recursion formula, Theorem 3.7, in \cite{MK-L:Back}.
\begin{proposition}
\[U^T_\star(\omega) = \underset{\Delta_\bullet(\omega)}\sum U^T_\star(\omega_1)
\cdot \oU^T_\star(\omega_2). \]
\end{proposition}

\begin{proof}
Given a forest $f\cdot t$ where $t$ is a tree. We will show
\[ \langle U_\star(ft), \omega \rangle = 
\underset{\Delta_\bullet(\omega)}\sum \langle ft, U^T_\star(\omega_1) \cdot
\oU^T_\star(\omega_2)\rangle.\]
We have:
\[ \langle U_\star(ft), \omega \rangle = 
\langle U_\star(f) \cdot U_\star(t), \omega \rangle. \]
Since concatenation and deconcatenation are dual maps, this is
\begin{align*}
= & \sum_{\Delta_\bullet(\omega)} \langle U_\star(f) \te U_\star(t),\omega_1 \te \omega_2 \rangle \\
= & \underset{\Delta_\bullet(\omega)}\sum 
\langle U_\star(f), \omega_1 \rangle \cdot \langle U_\star(t),
\omega_2 \rangle \\
= & \underset{\Delta_\bullet(\omega)}\sum \langle f, 
U^T_\star(\omega_1) \rangle \cdot \langle t, \oU^T_\star(\omega_2) \rangle.
\end{align*}
Since $\oU^T_\star(\omega_2)$ is a dual tree, this is:
\[ = \underset{\Delta_\bullet(\omega)}\sum \langle ft, U^T_\star(\omega_1) \cdot 
\oU^T_\star(\omega_2) \rangle. \]
\end{proof}

\subsubsection{The case of one free generator} \label{sssec:OneGen}
Now consider the case that $C = \langle \bullet \rangle$ is a one-dimensional
vector space. Recall the isomorphism $\psi: \cE_\bullet \iso T(\OT_\bullet^\gd)$ 
as algebras but the coproduct on this is different from 
$\cH_\bullet \iso T(\OT_\bullet^\gd)$.
To signify the difference, we denote the former by $T^\circ(\OT_\bullet^\gd)$.
It is the free algebra on the alphabet $a_\bullet(t)$ where the $t$ are
ordered trees.
Multiplication on $\cE_\bullet = Sym(P_\bullet^\gd)$ corresponds to
the shuffle product on $T^\circ(\OT_\bullet^\gd)$. 

The coproduct
\[ \cH_\bullet \mto{\Delta_\star} \cE_\bullet \te_\kk \cH_\bullet\]
may then by Subsubsection \ref{subsubsec:SymTe} be written as
\[ T(\OT_\bullet^\gd) \mto{\Delta_\star} T^\circ(\OT_\bullet^\gd) \te_\kk
T(\OT_\bullet^\gd) = K \te_\kk T(\OT_\bullet) . \]
%%GF
The two bialgebras $T(\OT_\bullet^\gd)$ and $T^\circ(\OT_\bullet^\gd)$
are said to be in cointeraction, a notion studied in 
\cite{EF:TwoInteract}, \cite{Ma-OrG}, \cite{Fo-Erh}, and
\cite{Fo-Chr}.%%

The element $U^t(\omega^{(2)})$ is in $C_K^* \iso K$. 
By the comment following \eqref{eq:SubUt} it is in 
\[ C_H^* = \Hom_\kk(\bullet,P_\bullet)^\gd \te_\kk \bullet^* \iso P_\bullet^\gd. \]
Then $U^t(\omega^{(2)})$ is simply the image of $\omega^{(2)}$ 
by the natural projection
$T(\OT_\bullet^\gd) \pil P_\bullet^\gd$. We may consider
$U^t(\omega^{(2)})$ as an
element of $K \iso T^\circ(\OT_\bullet^\gd)$ via the isomorphism $\psi$ above.
We are then using the Euler
idempotent map
\[ T(\OT_\bullet^\gd) \mto{\pi} T(\OT_\bullet^\gd) \iso T^\circ(\OT_\bullet^\gd), \] 
so that $U^t(\omega^{(2)}) = \pi(\omega^{(2)})$.

Let $B_+$ be the operation of attaching a root to a forest in order
to make it a tree. By a decorated shuffle $\shuffle$ below we
mean taking the shuffle product of the corresponding factors in
$K = T^\circ(\OT_\bullet^\gd)$. By the decorated $\cdot$ product
we mean concatenating the corresponding factors in $T(\OT_\bullet^\gd)$.
Then we may write
the recursion  as:
\begin{proposition}
\begin{align*} \Delta_\star(\omega) = & 
\shu_{13} \cdot_{24} 
\Delta_\star(\omega_1) \te \oU^T_\star(\omega_2) \\
= & \shu_{135} \cdot_{24} \Delta_\star(\omega_1) 
\te B^+( \Delta_\star(\omega_2^{(1)})) \te \pi(\omega_2^{(2)})
\end{align*}
\end{proposition}

%HMK
\section*{Acknowledgements}
We would like to thank Kurusch Ebrahimi-Fard,  Kristoffer F{\o}llesdal  and Fr\'{e}d\'{e}ric Patras for discussions on the topics of this paper.

\bibliographystyle{amsplain}
\bibliography{Bibliography}

\end{document}